\documentclass[10pt,reqno]{amsart}

\usepackage{amscd}
\usepackage{amsfonts}
\usepackage{amsmath}
\usepackage{amsthm}
\usepackage{amssymb}
\usepackage{enumitem}
\usepackage{setspace}
\usepackage{version}
\usepackage{tikz}
\usepackage{needspace}

\usepackage{hyperref}

\onehalfspace

\newcommand{\eps}[0]{\varepsilon}
\renewcommand{\bar}[1]{\overline{#1}}

\newcommand{\R}[0]{\mathbb{R}}
\newcommand{\C}[0]{\mathbb{C}}
\newcommand{\N}[0]{\mathbb{N}}
\newcommand{\T}[0]{\mathbb{T}}
\newcommand{\Z}[0]{\mathbb{Z}}
\newcommand{\Q}[0]{\mathbb{Q}}
\newcommand{\PP}[0]{\mathcal{P}}

\newcommand{\tttt}[0]{{$\times 2,\, \times 3$}}

\theoremstyle{plain}
\newtheorem{lemma}{Lemma}[section]
  \newtheorem{proposition}[lemma]{Proposition}
  \newtheorem{cor}[lemma]{Corollary}
  \newtheorem{thm}[lemma]{Theorem}
  \newtheorem*{thm*}{Theorem}

  \newtheorem*{fact*}{Fact}
  \newtheorem*{claim*}{Claim}
  \newtheorem{problem}[lemma]{Problem}

\theoremstyle{definition}
  \newtheorem{definition}[lemma]{Definition}

\theoremstyle{remark}
  \newtheorem{remark}[lemma]{Remark}

\begin{document}

\title{A solution to the pyjama problem}
\author{Freddie Manners}
\address{Mathematical Institute, Radcliffe Observatory Quarter, Woodstock Road, Oxford OX2 6GG}
\email{Frederick.Manners@maths.ox.ac.uk}

\begin{abstract}
  The ``pyjama stripe'' is the subset of $\R^2$ consisting of a vertical strip of width $2 \eps$ around every integer $x$-coordinate.  The ``pyjama problem'' asks whether finitely many rotations of the pyjama stripe around the origin can cover the plane.

  The purpose of this paper is to answer this question in the affirmative, for all positive $\eps$.  The problem is reduced to a statement closely related to Furstenberg's \tttt{} Theorem from topological dynamics, and is proved by analogy with that result.
\end{abstract}

\maketitle

\tableofcontents

\section{Introduction}
Fix some $\eps > 0$.  Let $E = E(\eps)$ denote a ``pyjama stripe'' viewed as a subset of $\C$:
\[
  E := \{ z \in \C \,:\, \Re(z)\, (\bmod\, 1) \in (-\eps, \eps) \}
\]
i.e.~a vertical strip of width $2 \eps$ around each integral $x$-coordinate.
(The term ``pyjama stripe'' is a reference to the resemblance of $E$ to the pattern on a pair of stripy pyjamas.)  For $\theta$ a unit complex number, write $E_\theta := \theta^{-1} E$ for the corresponding rotated pyjama stripe, i.e.~the set obtained by rotating $E$ by $\theta^{-1}$ about the origin.

\begin{problem}[Pyjama problem]
  \label{pyjama-problem}
  Do finitely many rotations of $E$ cover the plane?  I.e., does there exist a finite collection $\{\phi_1, \dots, \phi_k\}$ of unit complex numbers such that $\bigcup_{i} E_{\phi_i} = \C$?
\end{problem}

The goal of this paper is to answer this in the affirmative, for all positive $\eps$.

The problem was first stated in \cite{origins}, and was the subject of the more recent note \cite{mmr}, which inspired many of the ideas of this paper.  For a brief history and background of the problem, the reader is referred to \cite{mmr}[Section 1].

Our proof proceeds by reducing the pyjama problem to a statement (Lemma \ref{infinitary-rationality-lemma}) that resembles Furstenberg's \tttt{} Theorem \cite{furstenberg}[Part IV], but which -- as far as the author is aware -- does not appear in the literature.  The proof of this result follows that of the \tttt{} Theorem very closely, but faces significant technical complications, which bring into play a certain amount of $p$-adic analysis.  Where possible, we have tried to ensure the paper is nonetheless readable by someone without an extensive background in that field. 

\subsection{Acknowledgements}
The author is particularly grateful to Jonathan Lee for extensive discussions of these ideas; and also to Ben Green, Sean Eberhard, Przemys\l{}aw Mazur and Rudi Mrazovi\'{c} for their comments and scrutiny.

\section{Overview of the proof}
\subsection{Statement}
For concreteness, we will work from the outset with an explicit finite set of rotations.  At this stage, this choice is completely unmotivated; the motivation is spread over the next few sections.

We fix some notation.  Take $\PP_5 = 1 + 2 i$, $\PP_{13} = 2 + 3 i$; these are primes of $\Z[i]$ dividing $5$ and $13$ respectively.  Write $\theta_5 = \PP_5 / \bar{\PP_5}$ and $\theta_{13} = \PP_{13} / \bar{\PP_{13}}$, which in particular are unit complex numbers with rational coefficients (i.e.~elements of $\Q(i)$).

We restate the result in this setting.
\begin{thm}
  \label{pyjama-theorem}
  Let $\eps > 0$ be arbitrary. Define
  \[
    \Theta_N = \{ {\theta_5}^r {\theta_{13}}^s \,:\, r, s \in \Z,\, 0 \le r,\, s \le N \} \ .
  \]
  Also, for $n \ge 1$ define $\zeta_1^{(n)},\, \zeta_2^{(n)},\, \zeta_3^{(n)}$ to be any triple of unit complex numbers satisfying
  \[
    n \left(\zeta_1^{(n)} + \zeta_2^{(n)}\right) = \zeta_3^{(n)}
  \]
  and let $\Theta' = \Theta'(n, N) = \zeta_1^{(n)}\, \Theta_N \cup \zeta_2^{(n)}\, \Theta_N \cup \zeta_3^{(n)}\, \Theta_N$.

  Then there exist $n,\, N$ (depending on $\eps$) such that $\bigcup_{\theta \in \Theta'} E_\theta = \C$.
\end{thm}
\begin{remark} \ 
  \begin{enumerate}[label=(\roman*)]
    \item Finding a triple $\zeta_1^{(n)}, \zeta_2^{(n)}, \zeta_3^{(n)}$ as above is equivalent to finding a triangle (in $\C$) with side lengths $(n, n, 1)$; indeed, $\zeta_i^{(n)}$ correspond to the unit vectors in the direction of the edges of such a triangle.  By the converse to the triangle inequality, such a triangle exists, and it is straightforward to find one explicitly.
    \item Due to the infinitary nature of some of the methods used, the dependence of $n$ and $N$ (and thereby of $|\Theta'| = 3 (N+1)^2$) on $\eps$ given by our proof is completely ineffective.
    \item The significance of the numbers $5$ and $13$ is merely that they are primes of the form $4 k + 1$.  Any pair of such numbers would work, but again we fix these for concreteness.
  \end{enumerate}
\end{remark}

\subsection{Rational obstructions and periodic coverings}
It is noted in \cite{mmr}[Section 2] that \emph{periodic coverings} -- that is, choices of rotations $\phi_1, \dots, \phi_k$ such that all the sets $E_{\phi_i}$ have a common period lattice $\Lambda$ (of rank $2$) -- are attractive when approaching the pyjama problem.  Indeed, to check that $\bigcup_i E_{\phi_i} = \C$ for such a covering, it would suffice to check that some fundamental domain of $\Lambda$ was covered, which is essentially a finite task.

In particular, in the special case that all the rotations $\phi_i$ have rational coefficients -- i.e.~are elements of $\Q(i)$ -- we can choose some $D \in \Z[i]$ such that $D \phi_i \in \Z[i]$ for all $i$, and then deduce that $D\, \Z[i]$ is a period lattice.  Indeed, for any $z \in E_{\phi_i}$ and $r \in \Z[i]$ we have that $\phi_i\, z \in E$ (by definition), hence $\phi_i( z + D\, r) \in E$ (as $E$ is $\Z[i]$-periodic) and so we deduce $z + D\, r \in E_{\phi_i}$.  Note the sets $\Theta_N$ of rotations appearing in the statement of Theorem \ref{pyjama-theorem} fall within this special case.

However, it is also shown in \cite{mmr}[Theorem 2.3] that these periodic coverings cannot possibly cover all of $\C$ when $\eps$ is small (specifically, when $\eps < 1/3$).  The authors locate certain ``rational obstructions'', i.e.~points with a specific rational form which are guaranteed not to be covered by $\bigcup_i E_{\phi_i}$.  For example, if the $\phi_i$ are in $\Q(i)$ and $D$ is as above, then the point $\frac{1 + i}{2} D$ is not covered (for $\eps < 1/2$); this follows from the fact that if $\frac{r + s\, i}{t}$ is a unit vector (where $r$, $s$, $t$ are integers and the ratio is in its lowest terms) then necessarily $t$ is odd.

More generally, for any fixed $\eps > 0$ one can identify a finite set of tuples $(a, b, m)$ of integers such that, for any choice of $\phi_i \in \Q(i)$ and $D$ defined as above, the points $\frac{a + b\, i}{m} D$ are not covered. (Translates of these by $D\, \Z[i]$ are therefore also not covered.)  We term all of these ``rational obstructions''.  Their number increases without bound as $\eps \rightarrow 0$.

A simple example of a rational configuration with rotations $\{1, \theta_5\}$ is given in Figure \ref{fig-covering}, with the period square and the rational obstruction corresponding to $\frac{1 + i}{2} D$ indicated.
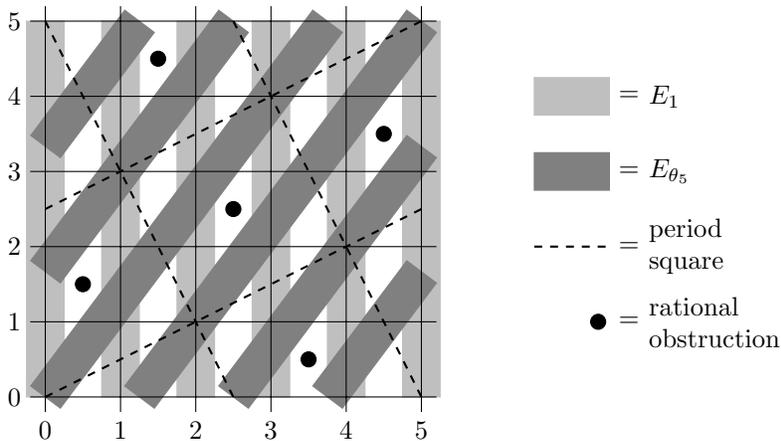
\begin{figure}
  \centering
  \begin{tikzpicture}
    \newcommand*{\bleed}{0.2}
    \newcommand*{\wvecx}{-0.2}
    \newcommand*{\wvecy}{0.15}

    \newcommand*{\slantrect}[4]{
      \draw[fill,black!50] (#1 + \wvecx, #2 + \wvecy) -- (#1 - \wvecx, #2 - \wvecy) -- (#3 - \wvecx, #4 - \wvecy) -- (#3 + \wvecx, #4 + \wvecy);
    }

    \foreach \i in {0, ..., 5} {
      \draw[fill, black!25] (\i - 0.25, 0) rectangle (\i + 0.25, 5);
    }

    \slantrect{0}{0}{15/4}{5}
    \slantrect{15/4}{0}{5}{5/3}
    \slantrect{0}{5/3}{5/2}{5}
    \slantrect{5/2}{0}{5}{10/3}
    \slantrect{0}{10/3}{5/4}{5}
    \slantrect{5/4}{0}{5}{5}

    \draw[thin,black,step=1.0] (-\bleed, -\bleed) grid (5 + \bleed, 5 + \bleed);
    \foreach \i in {0, ..., 5} {
      \draw[black] node [below] at (\i, -\bleed) {$\i$};
    }
    \foreach \i in {0, ..., 5} {
      \draw[black] node [left] at (-\bleed, \i) {$\i$};
    }

    \draw[thick,black,dashed] (0, 0) -- (5, 5/2);
    \draw[thick,black,dashed] (0, 5/2) -- (5, 5);
    \draw[thick,black,dashed] (0, 5) -- (5/2, 0);
    \draw[thick,black,dashed] (5/2, 5) -- (5, 0);

    \draw[fill,black] (5/2, 5/2) circle (0.1);
    \draw[fill,black] (3/2, 9/2) circle (0.1);
    \draw[fill,black] (9/2, 7/2) circle (0.1);
    \draw[fill,black] (1/2, 3/2) circle (0.1);
    \draw[fill,black] (7/2, 1/2) circle (0.1);

    \foreach \i in {1,...,4} {
      \draw (7.5, \i) node[right,black] {$=$};
    }

    \draw[fill, black!25] (6.5, 4.25) rectangle (7.5, 3.75);
    \draw (7.9, 4) node[right,black] {$E_1$};

    \draw[fill, black!50] (6.5, 3.25) rectangle (7.5, 2.75);
    \draw (7.9, 3) node[right,black] {$E_{\theta_5}$};

    \draw[black,thick,dashed] (6.5, 2) -- (7.5, 2);
    \draw (7.9, 2) node[right,black,align=left] {period \\ square};
    \draw[fill, black] (7.35, 1) circle (0.1);
    \draw (7.9, 1) node[right,black,align=left] {rational \\ obstruction};
  \end{tikzpicture}
  \caption{A periodic configuration $\{1,\,\theta_5\}$ showing the rational obstructions $\frac{1 + i}{2} D$.}
  \label{fig-covering}
\end{figure}
\\[\baselineskip]
Our strategy to overcome these obstructions is as follows.
\begin{itemize}
  \item We show that this is in some sense all that can go wrong.  That is, for suitably chosen $\phi_i \in \Q(i)$, the \emph{only} points in $\C \setminus \bigcup_i E_{\phi_i}$ are rational obstructions of the above form, or points reasonably close to them.  This result is termed the ``rationality lemma'' and is stated formally in Lemma \ref{rationality-lemma}.  Its proof will occupy the majority of the paper.  The rational rotations used correspond to the $\Theta_N$ in the statement of Theorem \ref{pyjama-theorem}.

  \item We then use a trick to deal with the missing points.  As this trick necessarily involves adding some irrational rotations to the set, we term it the ``irrational trick''; it corresponds to the $\zeta_i^{(n)}$ in the statement of Theorem \ref{pyjama-theorem}.  We note that this trick is essentially a variation of a technique used in \cite{mmr}[Theorem 3.1] to construct a covering with $\eps = 1/3 - 1/48$.  Section \ref{sec-irrational} describes this part of the argument.
\end{itemize}

The proof of the rationality lemma is by far the lengthier and harder part of the argument; we now turn to the strategy for proving that.

\subsection{The rationality lemma and topological dynamics}

The first stage in the proof of the rationality lemma is to deduce it from an analogous infinitary statement.  This process is very similar to (and inspired by) the infinitary reformulation of the pyjama problem itself in \cite{mmr}[Theorem 4.1], with a few notable differences -- most significantly that we are reformulating the rationality lemma rather than the pyjama problem.  The new version is Lemma \ref{infinitary-rationality-lemma}, the ``infinitary irrationality lemma''.

Our main reason for making this reformulation is that the new version bears a striking resemblance to Furstenberg's \tttt{} Theorem \cite{furstenberg}[Section IV].  Various extensions of Furstenberg's result are even closer, both formally and in spirit, to what we need: e.g.~work of Berend \cite{berend}, generalizing Furstenberg's result to arbitrary finite-dimensional connected compact abelian groups; and generalizations that allow an ``isometric direction'' pursued by Muchnik \cite{muchnik} and more recently by Wang \cite{wang}.  However, Lemma \ref{infinitary-rationality-lemma} does not follow from any of these results, or (as far as the author is aware) from any others in the published literature, although it could be seen as a case of a common generalization of all of them.  Here, we give a self-contained proof, by adapting Furstenberg's original argument -- or more accurately, a variant of it due to Boshernitzan \cite{boshernitzan} -- to our context.

\subsection{Layout of the paper}

Section \ref{sec-irrational} states the rationality lemma and deduces Theorem \ref{pyjama-theorem} from it, using the irrational trick.

Section \ref{sec-furstenberg} reproduces (for reference) a sketch proof of Furstenberg's \tttt{} Theorem, that will be used as a model for the proof of Lemma \ref{infinitary-rationality-lemma}. 

Section \ref{sec-infinitary} deals with the jump to the infinitary setting, and proves the relevant reductions.

Section \ref{sec-solenoid} provides some background on the limit object $\widehat{A}$ appearing in the infinitary rationality lemma, and states some standard properties of it.  We then state a number of less standard but essentially straightforward auxiliary results that we will need to run the proof of the \tttt{} Theorem; the lengthier and more technical proofs are consigned to Appendix \ref{appA}.

Finally, Section \ref{sec-rationality} reruns the proof from Section \ref{sec-furstenberg} in the new context, thereby proving the infinitary rationality lemma. 

\subsection{Notation}

We use $\Re(z)$ to denote the real part of a complex number.  In a metric space, $B_R(x)$ will denote the open ball of radius $R$ about $x$, and $\bar{B}_R(x)$ will similarly denote the closed ball.  The notation $x \ll y$ means there exists an absolute constant $C$ such that $x \le C y$, and $x \ll_{a_1, \dots, a_k} y$ means that $x \le C y$ where $C = C(a_1, \dots, a_k)$ can be a function of the variables $a_i$ only.

\section{The irrational trick}
\label{sec-irrational}

We state the rationality lemma formally.
\begin{lemma}[Rationality lemma]
  \label{rationality-lemma}
  Let $\Theta_N = \{ {\theta_5}^r {\theta_{13}}^s \,:\, 0 \le r,\, s \le N \}$ be as above, and let $\eps > 0$ be fixed.  For every $R$ there exists $D \in \Z[i]$ with $|D| \ge R$, and parameters $n \ll_\eps 1$, $N \ll_{\eps, R} 1$ such that the following holds: any point of $\C$ not contained in $\bigcup_{\theta \in \Theta_N} E_\theta$ is at distance at most $20$ from a ``very rational point'' $r \in \C$, in the sense that $r \in \frac{1}{n} D\, \Z[i] \setminus D \Z[i]$.
  
  Phrased another way, let $W = \C \setminus \bigcup_{\theta \in \Theta_N} E_\theta$ be the set of missing points; then $W$ is contained in the set
  \[
    W' = \bigcup \left\{ \bar{B}_{20}(r) \,:\, r \in \frac{1}{n} D \, \Z[i],\, r \notin D\, \Z[i] \right\} \subseteq \C \ .
  \]
\end{lemma}
\begin{remark} \hspace{0pt}
  \begin{enumerate}[label=(\roman*)]
    \item We see that $W'$ is doubly periodic with period $D \, \Z[i]$, as expected.
    \item It will be important that $n = n(\eps)$ is bounded independently of $R$.  This implies that the set $W'$ becomes very sparse as $D$ becomes large.
    \item In fact, $D$ will be some large power of $\bar{\PP_5} \bar{\PP_{13}}$.
  \end{enumerate}
\end{remark}

Once we know that the set of missing points is both very sparse and very structured, it is possible to combine several copies of $\Theta_N$ in such a way as to cover the plane.  There are likely to be many ways to do this; we give one below.

\begin{proof}[Proof of Theorem \ref{pyjama-theorem} given Lemma \ref{rationality-lemma}]
  Take $n \ll_\eps 1$ as in Lemma \ref{rationality-lemma}.  We choose $|D|$ in that lemma to be sufficiently large that the ball $\bar{B}_{40 n}(0)$, and its translates by $D\, \Z[i]$, are contained in $\bigcup_{\theta \in \Theta_N} E_\theta$.  Explicitly we can take any $|D| > 40 n^2 + 20 n$, as the nearest point of $W'$ to $0$ is at least a distance $|D| / n - 20$ away.

  Write $\Theta' = \Theta'(n, N)$ as in the statement, and abbreviate $\zeta_i^{(n)}$ to $\zeta_i$.  Suppose $z \in \C$ does not lie in $\bigcup_{\theta \in \Theta'} E_{\theta}$.  Equivalently, none of $\zeta_1 z$, $\zeta_2 z$, $\zeta_3 z$ lie in $\bigcup_{\theta \in \Theta_N} E_{\theta}$.  By Lemma \ref{rationality-lemma}, we have:
    \[
      \zeta_i z \in \bar{B}_{20}(r_i)
    \]
    for $i \in \{1, 2, 3\}$, for some $r_i \in \frac{1}{n} D \Z[i]$ (but $r_i \notin D \Z[i]$).

    In particular, $n \left(r_1 + r_2\right)$ lies in $D\, \Z[i]$, so $n (\zeta_1 z + \zeta_2 z)$ lies in some ball of radius $40\, n$ around a point of $D \, \Z[i]$.  By the choice of $D$, this point lies in $\cup_{\theta \in \Theta_N} E_\theta$.  But this point is also precisely $\zeta_3 z$, a contradiction.
\end{proof}

\section{Furstenberg's \tttt{} Theorem}
\label{sec-furstenberg}

Since it will be so central to our proof of the rationality lemma, we briefly review Furstenberg's \tttt{} Theorem and its proof here.

Recall that, for a topological dynamical system $(\Phi, X)$ (i.e.~$\Phi$ is a semigroup acting continuously on a compact metric space $X$) we say $Y \subseteq X$ is invariant if $\Phi(Y) \subseteq Y$.

So, let $\Phi = \{ 2^r 3^s :\, r, s \in \Z_{\ge 0} \} \subseteq \N$ be a semigroup under multiplication, acting on $\T = \R / \Z$ by multiplication.  The theorem states:

\begin{thm}[{\cite{furstenberg}[Section IV]}]
  \label{tttt}
  The only closed, invariant subsets $Y \subseteq \T$ with respect to this action are $\T$ itself, and certain finite sets of rationals in $\T$.
\end{thm}

This in fact holds whenever $\Phi$ is a multiplicative subsemigroup of $\N$ that is ``non-lacunary'', in the sense that $\Phi$ is not contained in any set of the form $\{a^r :\, r \in \Z_{\ge 0} \}$ for some $a \in \N$.  However, the case when $\Phi$ is the semigroup generated by $2$ and $3$ is essentially as hard as the general case.

Both Furstenberg's and Boshernitzan's proofs make use of the following lemma and its corollaries.  This is the point in the argument where the ``non-lacunary'' nature of $\Phi$ is used.

\begin{lemma}[{Related to \cite{furstenberg}[Lemma IV.1]}]
  \label{furstenberg-density}
  Fix $\delta > 0$. For sufficiently small $\eta > 0$ we have that $\eta \Phi \cap [0, 1)$ is $\delta$-dense in $[0, 1)$.
\end{lemma}
\begin{proof}
  It suffices to show that $\eta \Phi$ is $\delta/2$-dense on $[\delta/2, 1)$.  Taking logarithms, it is then also sufficient to show that $\log(\eta \Phi) \cap [\log(\delta / 2), 0]$ is $\delta'$-dense on $[\log(\delta / 2), 0]$, for some $\delta'$ sufficiently small in terms of $\delta$ (in fact $\delta'$ can be taken to be $\delta / 2$).

    Equivalently, this states that $\log \Phi = \{r \log 2 + s \log 3 \,:\, r, s \in \Z_{\ge 0} \}$ is $\delta'$-dense on $[\log(\delta / 2) + C, C]$ for $C = -\log \eta$ large enough.  Crucially, $\log 3 / \log 2$ is irrational, so we can choose $N$ such that $\{s \log 3 / \log 2 \pmod{1} \,:\, s \in \Z,\, 0 \le s \le N \}$ is $(\delta' / \log 2)$-dense on $\R / \Z$.  Then $\{r + s \log 3 / \log 2 \,:\, r, s \in \Z_{\ge 0}\}$ is $(\delta' / \log 2)$-dense on $[N \log 3 / \log 2, \infty)$ and so $\log \Phi$ is $\delta'$-dense on $[ N \log 3, \infty)$, which gives the result.
\end{proof}

\begin{cor}[{\cite{furstenberg}[Lemma IV.2]}]
  \label{furstenberg-zero-limit}
  Suppose $Y \subseteq \T$ is a closed invariant subset which has $0$ as a limit point.  Then $Y = \T$.  The same conclusion holds if $Y$ has just a rational limit point.
\end{cor}
\begin{proof}[Proof sketch]
  For any $\delta > 0$, we can choose a non-zero $t \in Y$ small enough that $|t| \Phi \cap [0, 1)$ is $\delta$-dense in $[0, 1)$, so \emph{a fortiori} the orbit $\Phi(t) \subseteq Y$ is $\delta$-dense in $\T$.  But $\delta$ was arbitrary and $Y$ was closed, so $Y = \T$.  The case of a rational limit point is similar, or can be deduced from this.
\end{proof}

The arguments in Furstenberg \cite{furstenberg} and Boshernitzan \cite{boshernitzan} now diverge; we broadly follow Boshernitzan.

\begin{proof}[Sketch proof of Theorem \ref{tttt}]
  Define $\Phi^{(m)} = \{ 2^{m r} 3^{m s} :\, r, s \in \Z_{\ge 0} \} \subseteq \Phi$, i.e.~the subsemigroup generated by $2^m$, $3^m$.  The above results hold also for $\Phi^{(m)}$ for any positive integer $m$ (as it is still ``non-lacunary'').

  If $Y$ is finite it is straightforward to see it consists of rationals.  Suppose $Y$ is infinite.  Let $Y'$ be the set of limit points of $Y$; it is another closed invariant subset, and non-empty as $Y$ is infinite.  If $Y'$ contains any rationals then $Y = \T$ by Corollary \ref{furstenberg-zero-limit}, so assume it doesn't.

  Fix $\delta > 0$.  Choose a finite, $\delta$-dense set of rationals in $\T$ whose denominators are coprime to $2$ and $3$; call them $\{a_0, \dots, a_{\ell - 1} \}$.  For ease of notation assume $a_0 = 0$.  Then for some $m$, $\Phi^{(m)}$ fixes all the $a_i$.

  For $0 \le k < \ell$, let $X_k = (Y' - a_0) \cap \dots \cap (Y' - a_k)$.  We claim each such $X_k$ is closed, $\Phi^{(m)}$-invariant and non-empty.  
  This will complete the proof: if we know that $X_{\ell - 1}$ is non-empty, then there exists some $x \in \T$ such that $\left\{ x + a_0, x + a_1, \dots, x + a_{\ell - 1}\right\} \subseteq Y'$.  In particular, $Y'$ is $\delta$-dense in $\T$, and hence so is $Y$; but $\delta > 0$ was arbitrary so $Y = \T$ (as it is closed).

  The proof of this claim is by induction on $k$; since $X_0 = Y'$ the case $k=0$ is clear.  Suppose $X_k$ is closed, invariant and non-empty.  We have $X_{k+1} = X_k \cap (Y' - a_{k+1})$ so this is certainly closed; and since $a_{k+1}$ fixed by $\Phi^{(m)}$, it is also $\Phi^{(m)}$-invariant.

  Since $X_k \subseteq Y'$ it contains only irrational points; so for any $x \in X_k$, the orbit $\Phi^{(m)}(x)$ is infinite and contained in $X_k$.  So $X_k$ is infinite.  Hence $(X_k - X_k)$ has $0$ as a limit point, and is closed and $\Phi^{(m)}$-invariant.  So $X_k - X_k = \T$ (by the corollary), and in particular there exist $x, y \in X_k$ such that $y - x = a_{k+1}$.  So $x = y - a_{k+1} \in Y' - a_{k+1}$ and $x \in X_k$, so $x \in X_{k+1}$ as required.
\end{proof}

\section{An infinitary reformulation of the rationality lemma}
\label{sec-infinitary}

\subsection{Introductory remarks}
We have two main motivations for considering infinitary versions of the rationality lemma.
\begin{itemize}
  \item The statement of Lemma \ref{rationality-lemma} involves four interdependent parameters ($\eps$, $D$, $n$, $N$), all of which can be successfully eliminated by passing to a suitable limit structure.  As is typical, this makes the statements and proofs significantly cleaner at the expense of sacrificing quantitative control.
  \item More importantly, passing to the limit exposes the connection with topological dynamics and Furstenberg's proof of the $\times 2,\, \times 3$ Theorem. This link is much more obscure in the finitary setting.
\end{itemize}
We are not aware of any reason in principle preventing the whole argument from being finitized to obtain effective bounds on $N$ and $n$ in terms of $\eps$.  However, the bounds obtained in this way would likely be of very poor quality (i.e.~at least of tower-type).

\subsection{Dual groups and a limit space}

In \cite{mmr}, the authors observe that the sets $E_\theta$ have a natural interpretation in terms of characters on $\C$.  

The continuous characters on $\C$ all have the form:
\begin{align*}
  \chi_{z} \,:\, \C & \rightarrow \T \\
                       w & \mapsto \Re\left(z\, w \right)\, (\bmod{1})
\end{align*}
where $z \in \C$ is some complex number.  So, simply by unravelling the definitions we can equivalently characterize the pyjama stripe $E_\theta$ by the formula:
\[
  E_\theta = \{ z \in \C \,:\, \chi_z(\theta) \in (-\eps, \eps) \} \ .
\]
The interesting feature of this definition is that it makes sense for \emph{any} character on $\C$, continuous or not.  Writing $\widehat{\C}$ for the group of all characters (not necessarily continuous) on $\C$ -- i.e.~the Bohr compactification -- we can similarly define:
\[
  E_\theta^\ast = \{ \chi \in \widehat{\C} \,:\, \chi(\theta) \in (-\eps, \eps) \} \subseteq \widehat{\C} \ .
\]
The notable properties of this situation are that $\widehat{\C}$ -- being the Pontryagin dual of $\C$ endowed with the discrete topology -- is compact with respect to its natural topology (the topology of pointwise evaluation); and moreover $E_\theta^\ast$ is an open set in this topology.  So, in \cite{mmr} the authors conclude that the pyjama problem is equivalent to the following infinitary question:
\begin{problem}[{Infinitary pyjama problem; \cite{mmr}[Lemma 4.1]}]
  \label{infinitary-pyjama-problem}
  Let $\eps > 0$.  Is there a set $\Phi$ of unit complex numbers -- which need not be finite and without loss of generality may be taken to be the whole unit circle -- such that $\bigcup_{\phi \in \Phi} E_\phi^\ast$ is all of $\widehat{\C}$?
\end{problem}
Deducing the pyjama problem from this is straightforward.  If such a set $\Phi$ exists, then by compactness there is actually a finite set $\Phi = \{\phi_1, \dots, \phi_k\}$ with the same property, i.e.:
\[
  \bigcup_i E_{\phi_i}^\ast = \widehat{\C} \ .
\]
Now, intersecting both sides of this equation with the set $\{ \chi_z \,:\, z \in \C \}$ of continuous characters gives
\[
  \bigcup_i E_{\phi_i} = \C
\]
as required.  The converse direction is less straightforward, and essentially follows from the fact that the continuous characters are dense in $\widehat{\C}$.

We refer to the space $\widehat{\C}$ as the \emph{limit space} being used in the arguments of \cite{mmr}.

\subsection{The smaller limit space $\widehat{A}$}

One downside of this approach, in our opinion, is that this limit space $\widehat{\C}$ is very large (for instance, non-metrizable) and generally difficult to work with.

However, we observe that we only ever need to know the value of a given $\chi \in \widehat{\C}$ \emph{at points of $\Phi$}.  So, the infinitary pyjama problem is well-defined on the quotient of $\widehat{\C}$ obtained by restricting to $\Phi$.  If $\Phi$ is very large, this doesn't help us very much; but if $\Phi$ is chosen very conservatively, this quotient might be comparatively manageable.

Specifically, suppose we work with the set of rotations
\[
  \Theta  = \Theta_\infty = \{ {\theta_5}^r {\theta_{13}}^s \,:\, r,\, s \in \Z_{\ge 0} \} = \bigcup_{N \ge 1} \Theta_N \ .
\]
We write $A$ for the subring of $\Q(i)$ generated by $\Z[i]$, $1/\bar{\PP_5}$ and $1/\bar{\PP_{13}}$, denoted $A = \Z[i][1 / \bar{\PP_5}, \, 1 / \bar{\PP_{13}}]$.  To put it another way, $A$ consists of those elements of $\Q(i)$ whose denominators are of the form $\bar{\PP_5}^r \bar{\PP_{13}}^s$.

In particular $\Theta \subseteq A$ and so it suffices to consider the quotient of $\widehat{\C}$ obtained by restricting the characters to $A$.  Equivalently, that space is precisely the Pontryagin dual $\widehat{A}$, where $A$ is given the discrete topology.

By contrast to $\widehat{\C}$, the space $\widehat{A}$ (again with the topology of pointwise evaluation) is reasonably tame.\footnote{It is perhaps worth noting that the heuristic statements ``we work with rational rotations so that the state space (i.e.~fundamental domain in $\C$) is finite'' and ``we work with rotations in $A$ so that the state space (i.e.~$\widehat{A}$) is reasonably tame'', are very closely related.} In general terms, it is metrizable and second-countable; but more importantly its structure can be understood entirely concretely as a certain quotient of a finite product of $\C$ and some $p$-adic fields.  We will return to this viewpoint in Section \ref{sec-solenoid}.

Of course, since we know that rotations in $\Q(i)$ cannot cover $\C$, the analogue of Problem \ref{infinitary-pyjama-problem} for $\widehat{A}$ is false.  Our aim is instead to transfer the rationality lemma (Lemma \ref{rationality-lemma}) to a statement on $\widehat{A}$.  A first attempt is given below.

Note that continuous characters on $\C$ are certainly characters of $A$; that is, there is a map
\begin{align*}
  \jmath_\C \,:\, \C & \rightarrow \widehat{A} \\
                   z & \mapsto {\chi_z} |_A \ .
\end{align*}
We write $\bar{B}^\C_R(0) \subseteq \widehat{A}$ for the image $\jmath_\C(\bar{B}_R(0))$ of the closed ball of radius $R$ about $0$ in $\C$, under $\jmath_\C$.  For $x \in \widehat{A}$, similarly define $\bar{B}^\C_R(x) = x + \bar{B}^\C_R(0)$.

\begin{lemma}[Infinitary rationality lemma, first formulation]
  \label{infinitary-rationality-lemma-v1}
  Pick $\eps > 0$, and let $\Theta$, $A$ be defined as above.  Treat
  \[
    E_\theta^\ast = \{ x \in \widehat{A} \,:\, x(\theta) \in (-\eps, \eps) \}
  \]
  now as a subset of $\widehat{A}$.  Define $W^\ast = \widehat{A} \setminus \bigcup_{\theta \in \Theta} E_\theta^\ast$ (the missing points).  Then there exists $n \ll_\eps 1$ such that $W^\ast$ is contained in the set:
  \[
    {W'}^\ast = \left\{ w \in \widehat{A} \,:\, n\, w \in \bar{B}_{10 n}^\C(0)  \right\}
  \]
  In other words, every point $w$ of $W^\ast$ is ``close to'' an $n$-torsion point of $\widehat{A}$ in the sense that they differ by the image under $\jmath_\C$ of a small complex number.
\end{lemma}

Note that, in this first pass reformulation, the parameters $N$ and $D$ have been eliminated but $\eps$ and $n$ have not.

The deduction of Lemma \ref{rationality-lemma} from this is somewhat less straightforward than in the case of Problem \ref{infinitary-pyjama-problem}, but is nonetheless essentially formal.

\begin{proof}[Deduction of Lemma \ref{rationality-lemma} from Lemma \ref{infinitary-rationality-lemma-v1}]
  Let $m \in \N$ and $\delta > 0$ be parameters to be specified later.  Then the set
  \[
    U_0 = \left \{ x \in \widehat{A} \,:\, x(i^\ell {\theta_5}^r {\theta_{13}}^s) \in (-\delta, \delta) \ \forall \ 0 \le r, s \le m,\, \ell \in \{0, 1\} \right \}
  \]
  is a standard open neighbourhood of $0$ in $\widehat{A}$, as the latter has the topology of pointwise evaluation.

  Define
  \[
    U = \bigcup_{w \in {W'}^\ast} (w + U_0) \subseteq \widehat{A}
  \]
  which is therefore an open neighbourhood of $W'^\ast$.  Assuming Lemma \ref{infinitary-rationality-lemma-v1}, we have that
  \[
    U \cup \bigcup_{\theta \in \Theta} E_\theta^\ast = \widehat{A} \ .
  \]
  By compactness, and since all these sets are open, there exists $N$ such that
  \[
    U \cup \bigcup_{\theta \in \Theta_N} E_\theta^\ast = \widehat{A} \ .
  \]
  Taking the preimage of both sides under $\jmath_\C$, we obtain
  \[
    \jmath_\C^{-1}(U) \cup \bigcup_{\theta \in \Theta_N} E_\theta = \C
  \]
  and rearranging this we get
  \[
    \C \setminus \bigcup_{\theta \in \Theta_N} E_\theta \subseteq \jmath_\C^{-1}(U) \ .
  \]
  Hence it suffices to verify that this mysterious set $\jmath_\C^{-1}(U)$ is contained in the set $W'$ from the statement of Lemma \ref{rationality-lemma}, for appropriate values of the parameters.  This is essentially just a case of unravelling the definitions.

  Suppose $z \in \jmath_\C^{-1}(U)$.  That means there exists $w \in {W'}^\ast$ and $u \in U_0$ such that $\chi_z = w + u$.  So, $n\, \chi_z = n w + n u$; but we know that $n w = \chi_{n t}$ for some $t \in \C$ with $|t| \le 10$.  So, $n u = \chi_{n z - n t}$.

  By the definition of $U_0$, we have $u(i^\ell {\theta_5}^r {\theta_{13}}^s) \in (-\delta, \delta)$ for all $0 \le r, s \le m$ and $\ell \in \{0, 1\}$.  So, $n u(i^\ell {\theta_5}^r {\theta_{13}}^s) \in (-n \delta, n \delta)$ for all such $r, s, \ell$ and hence
  \[
    \Re \left( n (z - t)\, i^\ell {\theta_5}^r {\theta_{13}}^s \right) \ (\bmod \, 1) \in (-n \delta, n \delta)
  \]
  for all such $r, s, \ell$, which implies that $n (z - t)\, {\theta_5}^r {\theta_{13}}^s$ is within distance $2 n \delta$ of a point of $\Z[i]$ for all $0 \le r, s \le m$.

  We isolate the following elementary claim.
  \begin{claim*}
    Write $D = \bar{\PP_5}^a \bar{\PP_{13}}^b$ for some $a, b \ge 0$. Fix some $x \in \C$, $\eta \le 1/100$ and suppose for all integers $r, s$ with $0 \le r \le a$, $0 \le s \le b$ we have that ${\theta_5}^r {\theta_{13}}^s x$ is within distance $\eta$ of a point of $\Z[i]$.  Then $x$ is within distance $\eta$ of a point of $D \Z[i]$.
  \end{claim*}
  \begin{proof}[Proof of claim]
    We proceed by induction on $a$ and $b$.  The base case $a = 0$, $b = 0$ is by assumption. If $a > 0$, we deduce by inductive hypothesis that $x$ and $\theta_5\, x$ both lie within $\eta$ of points of $\bar{\PP_5}^{a-1} \bar{\PP_{13}}^b \Z[i]$; call these points $y_1$, $y_2$ respectively.  Now, $|\theta_5 y_1 - y_2| \le 1/50$ and both $\theta_5 y_1$ and $y_2$ lie in $\bar{\PP_5}^{a-2} \bar{\PP_{13}}^b \Z[i]$, so $y_2 = \theta_5 y_1$ (as distinct points of $\bar{\PP_5}^{-1} \Z[i]$ are a distance at least $1/\sqrt{5}$ apart).  Hence $y_1 = {\theta_5}^{-1} y_2 \in \bar{\PP_5}^a \bar{\PP_{13}}^b \Z[i]$ as required.

    The induction step on $b$ is analogous.
  \end{proof}

  Applying the claim to $n (z - t)$ and setting $\delta = 1/(200 n)$, we deduce that $n (z - t)$ is within $1/100$ of a point of $D \Z[i]$ where $D = \bar{\PP_5}^m \bar{\PP_{13}}^m$; so $z$ is certainly within $20$ of a point $k \in \frac{1}{n} D \Z[i]$ (recalling $|t| \le 10$).  This means $z$ lies in $W'$ -- \emph{unless} the point $k$ lies in $D \Z[i]$ itself, which was specifically prohibited in the definition of $W'$.

  We check directly that this last case doesn't happen.  Indeed, suppose $z = k + t'$ for $k \in D\, \Z[i]$ and $|t'| \le 20$.  Since ${\theta_5}^r \ne 1$ for all $r \ge 0$, the map $\tau \mapsto \theta_5 \tau$ is an aperiodic rotation on the unit circle, and so by standard arguments we have that $\{ {\theta_5}^r\, t' \,:\, 0 \le r \le M \}$ is $\eps/2$-dense on the circle $\{z \in \C \,:\, |z| = |t'|\}$ for some $M \ll_\eps 1$. In particular there is some $0 \le r \le M$ such that ${\theta_5}^r\, t' \in E$.  Assuming $m \ge M$ (which we can), we get ${\theta_5}^r\, k \in \Z[i]$ and hence ${\theta_5}^r\, z \in E$.

  Noting that $m$ was allowed to be arbitrarily large, we get $|D|$ arbitrarily large (independently of $n$ and $\eps$), thereby completing the proof.
\end{proof}

\subsection{A dynamical reformulation}

Although Lemma \ref{infinitary-rationality-lemma-v1} puts the rationality lemma in a compact setting, the link with topological dynamics and the \tttt{} Theorem is not yet clear.  Our second (and final!) reformulation requires a simple but important shift in perspective.

Rather than rotating the pyjama stripe $E_1^\ast$ to $E_\theta^\ast$ and asking whether the rotated copies cover most of $\widehat{A}$, we keep the stripe fixed and rotate the \emph{points of $\widehat{A}$ itself}.  We then ask which points have rotations landing in the standard stripe $E_1^\ast$.

Formally, we define an action of $\Theta$ on $\widehat{A}$ by
\begin{align*}
  \rho \,:\, \Theta \times \widehat{A} &\rightarrow \widehat{A} \\
             (\theta, \chi) & \mapsto \left(a \mapsto \chi(\theta a) \right) \ .
\end{align*}
(We normally omit the $\rho$ and write $\theta(\chi)$ to refer to this action.)

This is a continuous, multiplicative semigroup action; i.e.~$(\Theta, \widehat{A})$ can be viewed as a topological dynamical system.  We can therefore use the terminology of invariant subsets etc.~in this context.

Note that
\begin{align*}
  E_\theta^\ast & = \left\{ \chi \in \widehat{A} \,:\, \theta(\chi) \in E_1^\ast \right\} \\ 
                & = \rho(\theta)^{-1} (E_1^\ast)
\end{align*}
where this once we use $\rho(\theta)^{-1}$ explicitly to denote the preimage, thereby avoiding confusion of $\theta^{-1}$ and $1/\theta$.  (Since $1/\theta \notin A$, its action on $\widehat{A}$ is not well-defined.)

We are now in a position to state the full infinitary rationality lemma.  Recall that $\chi \in \widehat{A}$ is a map $A \rightarrow \T$, so for $Y \subseteq \widehat{A}$ and $r \in A$, $Y(r) = \{\chi(r) \,:\, \chi \in Y\}$ denotes a subset of $\T$, the image of $Y$ under the evaluation map at $r$.
\begin{lemma}[Infinitary rationality lemma]
  \label{infinitary-rationality-lemma}
  As above, treat $(\Theta,\, \widehat{A})$ as a topological dynamical system.  If $Y \subseteq \widehat{A}$ is a closed, $\Theta$-invariant set such that the image of the evaluation map $Y(1) = \{\chi(1) \,:\, \chi \in Y\}$ is not all of $\T$, then $Y$ is contained in a set of the form
  \[
    U = \bigcup_{i=1}^k \bar{B}_{1/2}^\C(u_i)
  \]
  where $u_1, \dots, u_k$ is a finite collection of torsion points of $\widehat{A}$.
\end{lemma}

\begin{remark}
  \hspace{0pt}
  \begin{enumerate}[label=(\roman*)]
    \item The approximate dictionary with the \tttt{} Theorem is then:
      \begin{align*}
        \widehat{A} & \longleftrightarrow \T \\
        \left\{{\theta_5}^r {\theta_{13}}^s \,:\, r, s \in \Z_{\ge 0} \right\}  = \Theta & \longleftrightarrow \Phi = \left\{2^r 3^s \,:\, r, s \in \Z_{\ge 0} \right\} \\
                                                     Y(1) = \T & \longleftrightarrow Y = \T \\
                             Y \text{ is contained in some } U & \longleftrightarrow Y \text{ is a finite set of rationals}
      \end{align*}
      We note that the latter two are weaker than the ``obvious'' analogues, namely ``$Y = \widehat{A}$'' and ``$Y$ is a finite set of torsion points''.  We know of no reason why the result could not be strengthened to give $Y = \widehat{A}$ as one part of the dichotomy; but we do not need this strengthening, and our proof does not provide it.  By contrast, the form of $U$ \emph{cannot} be strengthened to ``$U$ is a finite set of torsion points'', as can be seen by taking, say, $Y = \bar{B}_{1/10}^\C(0) \subseteq \widehat{A}$.
    \item For all this to make sense, it was crucial that $\Theta$ was a multiplicative semigroup.  This forms part of the motivation for the choice of $\Theta$ and by extension of the $\Theta_N$.
    \item While the constants $10$ and $20$ appearing in $\bar{B}_{10n}^\C$ and $\bar{B}_{20}^\C$ in Lemmas \ref{infinitary-rationality-lemma-v1} and \ref{rationality-lemma} respectively were chosen fairly arbitrarily, the constant $1/2$ appearing here is best possible.  However, the value of this constant is never important, so the reader can mentally substitute a worse value if they wish.
    \item Notice that all the parameters (i.e.~$n$, $\eps$ etc.) appearing in Lemma \ref{rationality-lemma} have now been eliminated.
  \end{enumerate}
\end{remark}

Having seen how Lemma \ref{infinitary-rationality-lemma} relates to the \tttt{} Theorem, we now see how it relates to Lemma \ref{infinitary-rationality-lemma-v1} by deducing the latter from it.  This is very straightforward.

\begin{proof}[Deduction of Lemma \ref{infinitary-rationality-lemma-v1} from Lemma \ref{infinitary-rationality-lemma}]
  Consider the closed, invariant subset of $\widehat{A}$ given by
  \begin{align*}
    W^\ast &= \bigcap_{\theta \in \Theta} \rho(\theta)^{-1} \left(\widehat{A} \setminus E_1^\ast \right) \\
           &= \widehat{A} \setminus \bigcup_{\theta \in \Theta} \rho(\theta)^{-1} E_1^\ast \\
      &= \widehat{A} \setminus \bigcup_{\theta \in \Theta} E_\theta^\ast
  \end{align*}
  (closed because $E_1^\ast$ is open).  Then $W^\ast(1) \ne \T$ since $W^\ast \cap E_1^\ast = \emptyset$, i.e.~$W^\ast(1) \subseteq [\eps, 1 - \eps]$.  Hence $W^\ast \subseteq U$ for some $U$ of the specified form, and taking $n$ such that $n\, u_i = 0 \ \forall i$, this implies the conclusion of Lemma \ref{infinitary-rationality-lemma-v1}.
\end{proof}

To recap: in order to prove Theorem \ref{pyjama-theorem} (the pyjama problem) it now suffices to prove Lemma \ref{infinitary-rationality-lemma} using the analogy with the proof of Theorem \ref{tttt}.

\begin{remark}
  Dynamical systems of the form $(\Theta, \widehat{A})$ are by no means new. Such objects are referred to as ``$S$-integer dynamical systems'' in \cite{cew}, and similar spaces are studied in \cite{berend}, as mentioned in the introduction.
\end{remark}

\subsection{Remarks on the choice of $\Theta$}
The motivation for the choice of the rotations $\Theta$, $\Theta_N$ and ultimately $\Theta'(n, N)$ (see Theorem \ref{pyjama-theorem}) should now be complete, and we offer a few remarks.

We noted above that $\Theta$ needs to be a semigroup.  All we really needed about the sets $\Theta_N$ was that their union was $\Theta$; but taking $\Theta_N$ to be a long multidimensional progression on the generators is the most natural choice.  Once we have chosen $\Theta_N$, the choice of $\Theta'(n, N)$ is fixed by the use of the irrational trick (see Section \ref{sec-irrational}).

We could in principle have used a larger semigroup in place of $\Theta$, such as the set of all unit norm elements of $\Q(i)$, which clearly contains $\Theta$.  Although this cannot hurt in some sense -- adding more rotations doesn't make the problem harder -- we would be forced to pass to a larger limit structure, e.g.~$\widehat{\Q(i)}$.  This is problematic for two reasons.
\begin{enumerate}[label=(\roman*)]
  \item Working with this larger, more complicated space introduces yet more technical hurdles than already exist in Section \ref{sec-solenoid}.
  \item Trying to get any benefit out of the extra rotations brings into play non-trivial questions about the distribution of rational points on the unit circle, or about the distribution of primes.
\end{enumerate}

In the other direction, we could not have used a much smaller semigroup.  Indeed, the appropriate results are false for the smaller set $\{ {\theta_5}^r \,:\, r \ge 0 \}$; by analogy with Theorem \ref{tttt}, we might say this set is ``lacunary'' in some sense.

In other words, $\Theta$ is the simplest example that works, in much the same way that $\{2^r 3^s \,:\, r, s \in \Z_{\ge 0} \}$ is the simplest case that works in the \tttt{} Theorem.

\section{Technical results about $\widehat{A}$}
\label{sec-solenoid}

\subsection{Introductory remarks}
The remainder of the paper is dedicated to the proof of Lemma \ref{infinitary-rationality-lemma} by analogy with the \tttt{} Theorem.  However, there are two sources of technical complication.
\begin{itemize}
  \item The proof of the \tttt{} Theorem required a certain amount of detailed knowledge of the compact group $\T$: e.g.~what its torsion points are; what its finite orbits under $\Phi$ are, and so forth.  Transferring the proof to $\widehat{A}$ requires a similar degree of knowledge of $\widehat{A}$.  Mostly these facts are not hard, but require slightly more justification in this less familiar setting.

    This is most pronounced when it comes to proving an analogue of the ``density estimate'' of Lemma \ref{furstenberg-density}.  Recall this was fairly straightforward after taking logarithms.  Essentially the same proof works; but the act of ``taking logarithms'' on $\widehat{A}$ requires a very detailed knowledge of the structure of that space in terms of $p$-adic fields.

    Similar difficulties are overcome in \cite{berend}, although we have not followed that author's approach very closely.
  \item The statement of Lemma \ref{infinitary-rationality-lemma} necessarily involves the ``small complex balls'' $\bar{B}_{1/2}^\C$ that appear in the definition of $U$.  This means that ``small complex errors'' have to be carried through the entire argument, making some of the statements much less transparent.
\end{itemize}

In this section, we quote a number of facts about the detailed structure of $\widehat{A}$, and go on to state some auxiliary results that address these technical difficulties.  These are typically translations of trivial results on $\T$, and where possible we will make the connection explicit.

\subsection{A description of $\widehat{A}$}

Recall that $A$ is the subring of $\Q(i)$ generated by $\Z[i]$, $1 / \bar{\PP_5}$ and $1 / \bar{\PP_{13}}$.  That is, it consists of all elements of $\Q(i)$ whose denominators are of the form $\bar{\PP_5}^r \bar{\PP_{13}}^s$.

So far we have used nothing about $\widehat{A}$ other than its definition and topology.  We now describe the characters on $A$ explicitly.

This material is fairly standard and can be found in many places in the literature.  The application of harmonic analysis to number fields is perhaps most famously associated with \cite{tate}, and much of what follows can be extracted from that paper.  The exact form of the results we need appears in \cite{cew}[Section 3].  The author has found \cite{conrad} to be a very good introduction to these ideas.

However, to keep things as approachable as possible to those unfamiliar with this material, this subsection takes on an expository flavour, while leaving rigorous details to the references.

\subsubsection{A model case: $\widehat{\Z[1/2]}$}

A slightly simpler case to consider is the dual group of $\Z[1/2]$, the ring of dyadic rationals treated as a discrete additive group.  The dual $\widehat{\Z[1/2]}$ is often referred to as a \emph{solenoid} or \emph{the $2$-solenoid}.

Clearly any continuous character on $\R$ restricts to one on $\Z[1/2]$.  The continuous characters on $\R$ can be written
\begin{align*}
  \chi_x \,:\, \R &\rightarrow \T \\
                y & \mapsto \{x\, y\}
\end{align*}
where $\{\cdot\}$ denotes the fractional part; so, analogously to $\jmath_\C \,:\, \C \rightarrow \widehat{A}$ above, we get a map
\begin{align*}
  \jmath_\R \,:\, \R &\rightarrow \widehat{\Z[1/2]} \\
                   x &\mapsto \chi_x|_{\Z[1/2]} \ .
\end{align*}
(It is not hard to verify that in fact $\jmath_\R$ is injective.)

As well as being contained in the reals, $\Z[1/2]$ is also contained in the $2$-adic rationals $\Q_2$; indeed, $\Q_2$ is the completion of $\Z[1/2]$ with respect to the $2$-adic metric.  So analogously, any continuous character of $\Q_2$ restricts to a character of $\Z[1/2]$.  The continuous characters of $\Q_2$ are all of the form:
\begin{align*}
  \chi_a \,:\, \Q_2 &\rightarrow \T \\
                  b &\mapsto \{a\, b\}_2
\end{align*}
where $a \in \Q_2$ and $\{x\}_2$ is the ``$2$-adic fractional part'', defined as the unique dyadic rational in $[0, 1)$ such that $x - \{x\}_2 \in \Z_2$ (the $2$-adic integers).  Hence we get another map
\begin{align*}
  \jmath_{\Q_2} \,:\, \Q_2 &\rightarrow \widehat{\Z[1/2]} \\
                   a &\mapsto \chi_a|_{\Z[1/2]} \ .
\end{align*}
which is also injective.  By combining the real and $2$-adic characters on $\Z[1/2]$, we get a map
\begin{align*}
  \jmath \,:\, \R \times \Q_2 &\rightarrow \widehat{\Z[1/2]} \\
                       (x, a) &\mapsto -\jmath_\R(x) + \jmath_{\Q_2}(a) \ .
\end{align*}
Now, $\jmath$ is not injective, and it is not too difficult to see that $\ker \jmath$ consists of pairs $(r, r)$, where $r \in \Z[1/2]$ is treated as a real number and a $2$-adic number respectively.  Less straightforward is the fact that $\jmath$ is actually \emph{surjective}.  Assuming that fact, we get the following conclusion.

\newcommand{\zhalfhat}[0]{$\widehat{\Z[1/2]}$}
\needspace{2\baselineskip}
\begin{proposition}[{Structure of \zhalfhat}]
  \hspace{0pt}
  \begin{enumerate}[label=(\roman*)]
    \item Consider the diagonal embedding
      \begin{align*}
        \imath^{\Delta} \,:\, \Z[1/2] &\rightarrow \R \times \Q_2 \\
                                    r &\mapsto (r, r) \ .
      \end{align*}
      Then $\imath^{\Delta}(\Z[1/2])$ is a discrete and co-compact subgroup of $\R \times \Q_2$, and $\jmath$ (as defined above) gives rise to an isomorphism of topological groups
      \[
        \tilde{\jmath} \,:\, (\R \times \Q_2) / \imath^{\Delta}(\Z[1/2]) \rightarrow \widehat{\Z[1/2]}
      \]
      where the left hand side is given the natural (product, quotient) topology coming from the usual topologies on $\R$ and $\Q_2$.  In particular, $\widehat{\Z[1/2]}$ is naturally a compact metric space.
    \item A fundamental domain for $\imath^{\Delta}(\Z[1/2])$ is given by $[0, 1) \times \Z_2$.  This does not give a rise to a topological isomorphism or an isomorphism of groups.  It is, however, possible to simplify the above quotient to $(\R \times \Z_2) / \imath^{\Delta}(\Z)$.
  \end{enumerate}
\end{proposition}

\subsubsection{Applying to $\widehat{A}$}

Much the same analysis carries over to $\widehat{A}$.  (We will overload notation from the $\Z[1/2]$ case.  Since we will never need to refer to $\Z[1/2]$ in the argument this should not cause too much confusion.)

We have already seen the complex characters in $\widehat{A}$ given by $\jmath_\C$.  To define the appropriate non-Archimedean characters, we have to consider the non-Archimedean fields obtained by completing $\Q(i)$ with respect to the $\bar{\PP_5}$- and $\bar{\PP_{13}}$-adic metrics.

In fact, these completions are canonically (and topologically) isomorphic to the more usual non-Archimedean fields $\Q_5$ and $\Q_{13}$ respectively.  The reason is that $\Q_5$ and $\Q_{13}$ already contain square roots of $-1$; sending $i$ to one of these gives the isomorphism, and which root to choose is specified by the choice of e.g.~$\bar{\PP_5}$ rather than $\PP_5$.  In the sequel, we will talk in terms of $\Q_5$ and $\Q_{13}$ rather than the completions of $\Q(i)$ by $\bar{\PP_5}$ and $\bar{\PP_{13}}$; though we will continue to refer to the absolute values $|\cdot|_{\bar{\PP_5}}$ and $|\cdot|_{\bar{\PP_{13}}}$ on $\Q(i)$ from time to time.

Running the analysis as before gives the following result.

\needspace{2\baselineskip}
\begin{proposition}[The structure of $\widehat{A}$]
  \label{solenoid-general-facts} \hspace{0pt} 
  \nopagebreak
  \begin{enumerate}[label=(\roman*)]
    \item There are canonically specified roots of $-1$, denoted $i_5 \in \Q_5$ and $i_{13} \in \Q_{13}$, such that the following holds.  The field $\Q(i)$ is embedded in $\Q_5$ and $\Q_{13}$ by the unique field homomorphisms
      \begin{align*}
            \imath_{\Q_5} \,:\, \Q(i) &\rightarrow \Q_5 \\
         \imath_{\Q_{13}} \,:\, \Q(i) &\rightarrow \Q_{13}
      \end{align*}
      given by taking $i$ to $i_p$ in each case.  Under these embeddings, the absolute values $|\cdot|_{\bar{\PP_p}}$ on $\Q(i)$ and $|\cdot|_p$ on $\Q_p$ agree (for $p \in \{5, 13\}$); e.g.~$|\imath_{\Q_5}(q)|_5 = |q|_{\bar{\PP_5}}$ for $q \in \Q(i)$, and similarly for $p=13$.\footnote{These absolute values certainly agree up to an arbitrary exponent, and we choose normalizations to make them agree exactly.}
      
    Also, $\Q(i)$ is embedded in $\C$ in the usual way, which we denote $\imath_\C$ for consistency.
    \item Under the diagonal embedding
      \begin{align*}
        \imath^{\Delta} \,:\, \Q(i) &\rightarrow \C \times \Q_5 \times \Q_{13} \\
                              q &\mapsto \left(\imath_\C(q), \imath_{\Q_5}(q), \imath_{\Q_{13}}(q) \right)
      \end{align*}
      the image $\imath^{\Delta}(A)$ of $A$ is discrete and co-compact with respect to the usual metric.
    \item In addition to $\jmath_\C \,:\, \C \rightarrow \widehat{A}$ defined above, we define
      \begin{align*}
        \jmath_{\Q_5} \,:\, \Q_5 &\rightarrow \widehat{A} \\
                               a &\mapsto \left(r \mapsto \{a  \cdot \imath_{\Q_5}(r)\}_5\right)
      \end{align*}
      where $\{\cdot\}_5$ is the $5$-adic fractional part as above, and similarly
      \begin{align*}
        \jmath_{\Q_{13}} \,:\, \Q_{13} &\rightarrow \widehat{A} \\ 
                                     a &\mapsto \left(r \mapsto \{a \cdot \imath_{\Q_{13}}(r)\}_{13}\right) \ .
      \end{align*}
      Combining these, we get a map
      \begin{align*}
        \jmath \,:\, \C \times \Q_5 \times \Q_{13} &\rightarrow \widehat{A} \\
                                         (z, a, b) &\mapsto -\jmath_\C(z) + \jmath_{\Q_5}(a) + \jmath_{\Q_{13}}(b) \ .
      \end{align*}
    \item The kernel of $\jmath$ is precisely $\imath^{\Delta}(A)$, and $\jmath$ is surjective, so $\jmath$ gives rise to a natural isomorphism of topological groups
      \[
        \tilde{\jmath} \,:\, (\C \times \Q_5 \times \Q_{13}) / \imath^{\Delta}(A) \rightarrow \widehat{A} \ .
      \]
    \item Consequently, $\widehat{A}$ is naturally a connected compact metric space, where the metric
      \[
        d_{\widehat{A}}(x, y) = \inf \left\{ |z|_\C + |a|_5 + |b|_{13} \,:\, (z, a, b) \in \C \times \Q_5 \times \Q_{13},\ \jmath(z, a, b) = x - y \right\}
      \]
      is also translation-invariant.
    \item \label{enum:bijection} A fundamental domain for $\imath^{\Delta}(A)$ is given by $[0, 1)^2 \times \Z_5 \times \Z_{13}$.  This does not give rise to a topological isomorphism or an isomorphism of groups.  However, it is possible to simplify the above quotient to $(\C \times \Z_5 \times \Z_{13}) / \imath^{\Delta}(\Z[i])$, though we shall not use this fact.
    \item \label{enum:action} The action $\rho$ of $\Theta$ on $\widehat{A}$ defined above, is equivalent (under $\jmath$) to
      \[
        \rho(\theta) (z, a, b) = (\imath_\C(\theta) z,\,  \imath_{\Q_5}(\theta) a,\,  \imath_{\Q_{13}}(\theta) b)
      \]
      i.e.~to multiplying by $\theta$ on each factor, combined with the appropriate embeddings.  This action in fact makes sense on the whole of $A$, not just $\Theta$.
  \end{enumerate}
\end{proposition}
\begin{proof}[Remarks on the proof]
  We have already given an overview of the approach to proving this.  The content is in parts (ii) and (iv), for which we refer the reader to \cite{cew}[Theorem 3.1] for a discussion of the rigorous details.  The remaining parts are easy consequences of these together with well-known facts from the theory of local fields.
  \renewcommand{\qedsymbol}{{}}
\end{proof}

With this as background, we move on to some auxiliary results.

\subsection{Torsion points and periodic points}

\begin{definition}
  For a positive integer $m$, we define $\Theta^{(m)} = \left\{ {\theta_5}^{r m} {\theta_{13}}^{s m} \,:\, r, s \in \Z_{\ge 0} \right\}$, a subsemigroup of $\Theta$.

  We say a point $x \in \widehat{A}$ is \emph{periodic} if $\Theta^{(m)}$ fixes $x$ for some $m$.
\end{definition}

Recall that, in the proof of the \tttt{} Theorem, we made an analogous definition of $\Phi^{(m)}$, and points fixed by $\Phi^{(m)}$ played an important role in the argument.  We implicitly used the fact that the periodic points in the context of $\T$ are precisely rationals whose denominator is coprime to $2$ and $3$.

A related -- and completely trivial -- fact is that the torsion points of $\T$ are precisely the rationals.  So, if $x \in \T$ is a torsion point then $2^r 3^s x$ is periodic for some $r, s \ge 0$.

In this subsection we transfer some of these results to $\widehat{A}$.  First we classify the periodic and torsion points.  Those proofs not given here can be found in Appendix \ref{appA}.

\begin{proposition}
  \label{periodic-points}
  Let $x \in \widehat{A}$, and take $(z, a, b) \in \C \times \Q_5 \times \Q_{13}$ any representative of $x$ (i.e.~$\jmath(z, a, b) = x)$.
  \begin{enumerate}[label=(\roman*)]
    \item \label{enum:torsion-equiv} The following are equivalent:
      \begin{enumerate}[label=(\alph*), ref=\alph*]
        \item \label{torsion-cond} $x$ is torsion;
        \item \label{torsion-rationality-cond} there exists $q \in \Q(i)$ such that $z = \imath_\C(q)$, $a = \imath_{\Q_5}(q)$ and $b = \imath_{\Q_{13}}(q)$; equivalently, $(z, a, b) \in \imath^{\Delta}(\Q(i))$;
        \item \label{orbit-finite-cond} the orbit $\Theta(x)$ is finite.
      \end{enumerate}
    \item \label{enum:periodic-equiv} The following are equivalent:
      \begin{enumerate}[label=(\alph*), ref=\alph*]
        \item \label{fixed-point-cond} $x$ is periodic;
        \item \label{periodic-rationality-cond} there is some $q \in \Q(i)$ such that $z = \imath_\C(q)$, $a = \imath_{\Q_5}(q)$, $b = \imath_{\Q_{13}}(q)$, and additionally $|q|_{\PP_5},\, |q|_{\PP_{13}} \le 1$.
      \end{enumerate}
  \end{enumerate}
\end{proposition}
Note the appearance of $|\cdot|_{\PP_5}$ and $|\cdot|_{\PP_{13}}$ in \ref{enum:periodic-equiv} (\ref{periodic-rationality-cond}); \emph{not} $|\cdot|_{\bar{\PP_5}}$ and $|\cdot|_{\bar{\PP_{13}}}$.

One useful application is the following corollary.

\begin{cor}
  \label{torsion-nearly-implies-periodic}
  If $x \in \widehat{A}$ is torsion then for some $\theta \in \Theta$, $\theta( x)$ is periodic.
\end{cor}
\begin{proof}
  Given a $q \in \Q(i)$ satisfying \ref{enum:torsion-equiv}(\ref{torsion-rationality-cond}), we can pick a $\theta \in \Theta$ to remove powers of $\PP_5$ and $\PP_{13}$ from the denominator.  Then $\theta\, q$ satisfies \ref{enum:periodic-equiv}(\ref{periodic-rationality-cond}).
\end{proof}

\subsection{Unions of complex balls}

By a ``complex ball' in $\widehat{A}$ we mean a set of the form $\bar{B}_R^\C(x)$, where we recall $\bar{B}_R^\C(0)$ is the image of the ball of radius $R$ about $0$ in the complex plane under the map $\jmath_\C$, and $\bar{B}_R^\C(x)$ is its translate by $x \in \widehat{A}$.

The proof of the \tttt{} Theorem made use of the following trivial facts.
\begin{enumerate}[label=(\roman*)]
  \item If $Y \subseteq \T$ is a finite $\Phi^{(m)}$-invariant set, then $Y$ consists entirely of rationals.
  \item If $x \in \T$ is irrational then its orbit $\Phi^{(m)}(x)$ is infinite.
\end{enumerate}
In adapting these to our context, we not only need to talk about torsion points of $\widehat{A}$ in place of rationals, but also introduce ``small complex errors'' into the hypotheses and conclusions, as discussed in the introduction to this section. Here is the resulting analogous statement.

\begin{proposition}
  \hspace{0pt}
  \label{complex-by-torsion-lemma}
  \begin{enumerate}[label=(\roman*)]
    \item \label{enum:cbf-implies-cbt} Suppose $Y \subseteq \widehat{A}$ is closed, $\Theta$-invariant, and contained in a finite union of complex balls (equivalently, of unit complex balls).  Then it is contained in a finite union of complex balls with centers at torsion points of $\widehat{A}$.

    \item \label{enum:cbf-iff-cpt} Let $x \in \widehat{A}$.  Then the orbit closure $\bar{\Theta(x)}$ is contained in a finite union of complex balls, if and only if $x$ has the form $x = y + \jmath_\C(z)$ where $y$ is torsion and $z \in \C$.
  \end{enumerate}
  Moreover, the above holds replacing $\Theta$ by $\Theta^{(m)}$.
\end{proposition}

\subsection{Non-Archimedean limits}

Recall the crucial result for proof of the \tttt{} Theorem (Corollary \ref{furstenberg-zero-limit}): if $Y \subseteq \T$ is closed, invariant and has $0$ as a limit point then $Y = \T$.

The naive generalization to $\widehat{A}$ -- i.e.~that if $Y \subseteq \widehat{A}$ is closed, invariant and has $0$ as a limit point then $Y = \widehat{A}$ -- is false, as again can be seen by taking $Y = \bar{B}_{1/10}^\C(0)$.

However, this counterexample is in some sense the only one.  The view taken is as follows: in the same way that the singleton $\{0\} \subseteq \T$ does not have $0$ as a limit point -- the constant $0$ sequence not being a valid way to approach $0$ -- similarly, $\bar{B}_{1/10}^\C(0)$ should not have $0$ as a ``proper'' limit point, because sequences approaching $0$ along a complex ball around $0$ should be likewise invalid.

This is formalized in the following definition.
\begin{definition}
  We say $x$ is a \emph{non-Archimedean limit point} of $Y \subseteq \widehat{A}$ if $x$ is a limit point of $Y \setminus B_{1/10}^\C(x)$;  equivalently, if there is a sequence $x_n \rightarrow x$ of points of $Y$ such that $x_n - x \notin B_{1/10}^\C(0)$ (in which case we say $x$ is a non-Archimedean limit of $x_n$).  Denote by $Y'$ the set of non-Archimedean limit points of $Y$.
\end{definition}
\begin{remark}
  The constant $1/10$ appearing in this definition is completely unimportant, and altering it does not affect the definition.
\end{remark}

The analogue of the above-mentioned result will have to wait (see Lemma \ref{zero-limit-lemma}); for now we record some basic facts about these non-Archimedean limits.

\begin{proposition}
  \label{limit-is-closed}
  For any $Y \subseteq \widehat{A}$, $Y'$ is closed.  If $Y$ is $\Theta^{(m)}$-invariant, then so is $Y'$.
\end{proposition}
This is, of course, analogous to the corresponding properties of the set of usual limit points of a set.
\begin{proof}
  We can write $Y'$ as:
  \[
    Y' = \bigcap_{x \in \widehat{A}} \left( \bar{Y \setminus B_{1/10}^\C(x)} \right)
  \]
  which is clearly closed.  Invariance is clear.
\end{proof}

Also recall that we used the fact that an infinite set in a compact space has a limit point.  Below is an analogue for non-Archimedean limits.

\begin{proposition}
  \label{infinite-limit-lemma}
  Suppose $Y \subseteq \widehat{A}$ is closed and not contained in a finite union of complex balls.  Then $Y'$ is non-empty, and $0 \in (Y - Y)'$.
\end{proposition}

\subsection{Density of $\C$, $\Q_5$ and $\Q_{13}$ in $\widehat{A}$}

It is a standard fact that the embedded copies $\jmath_\C(\C)$, $\jmath_{\Q_5}(\Q_5)$ and $\jmath_{\Q_{13}}(\Q_{13})$ of $\C$, $\Q_5$ and $\Q_{13}$ respectively in $\widehat{A}$, are all dense; equivalently, any $x \in \widehat{A}$ can be arbitrarily well approximated by complex characters, or by $5$-adic or $13$-adic characters.

We will only need this fact in the case of $\Q_5$ and $\Q_{13}$, and there we will actually need the following slight strengthening.

\begin{proposition}
  \label{strong-strong-approx}
  Suppose $J \le \Q_5^\times$ is a finite index multiplicative subgroup, and $H$ is some coset of $J$.  Then $\jmath_{\Q_5}(H) = \jmath(0, H, 0)$ is dense in $\widehat{A}$. The same holds symmetrically for $\Q_{13}$.
\end{proposition}

The proof is, as usual, in Appendix \ref{appA}, and on this occasion is slightly lengthy.  We mention one key step, namely the following related result about simultaneously approximating elements of $\C$, $\Q_5$ and $\Q_{13}$ by elements of $A$.

\begin{proposition}
  \label{strong-approx}
  Let $z \in \C$ and $b \in \Q_5$ be given.  For any $\delta > 0$, we can find an element $q \in A$ such that $|\imath_\C(q) - z|_\C,\, |\imath_{\Q_5}(q) - b|_5 \le \delta$.

  The same holds swapping $5$ and $13$.
\end{proposition}
\begin{proof}
  This is a special case of what is sometimes called the Strong Approximation Theorem (which is no more than the natural generalization of this statement); see \cite{cassels}[Chapter 10, Theorem 4.1].
\end{proof}

\subsection{Density of periodic points of $\widehat{A}$}

In the \tttt{} Theorem, we noted that for all $\delta > 0$ there exists an $m$ and a finite, $\delta$-dense subset of $\T$ consisting of points fixed by $\Phi^{(m)}$.  We now turn to the (completely unaltered) statement for $\widehat{A}$.

\begin{proposition}
  \label{many-periodic-points}
  For all $\delta > 0$ there exists an $m \in \Z_{>0}$ and a finite set $S \subseteq \widehat{A}$ which is pointwise fixed by $\Theta^{(m)}$, such that $S$ is $\delta$-dense in $\widehat{A}$.
\end{proposition}

We need the following result, very similar in flavour to Proposition \ref{strong-approx} above.  The number $7$ is essentially arbitrary here.
\begin{proposition}
  \label{strong-approx-7}
  Let $z \in \C$, $a \in \Q_5$ and $b \in \Q_{13}$ be given.  For any $\delta > 0$ we can find $q \in A[1/7]$ (say) such that $|\imath_\C(q) - z|_\C,\, |\imath_{\Q_5}(q) - a|_5,\, |\imath_{\Q_{13}}(q) - b|_{13} \le \delta$.  Equivalently, $\imath^{\Delta}(A[1/7])$ is dense in $\C \times \Q_5 \times \Q_{13}$, or if you prefer $\jmath(\imath^{\Delta}(A[1/7]))$ is dense in $\widehat{A}$.
\end{proposition}
\begin{proof}
  This is also a special case of Strong Approximation; again see \cite{cassels}[Chapter 10, Theorem 4.1].
\end{proof}

\begin{proof}[Proof of Proposition \ref{many-periodic-points}]
  We consider the sets $R_n = 7^{-n} A \subseteq \Q(i)$, and their images $S_n = \jmath\left(\imath^{\Delta}(R_n)\right) \subseteq \widehat{A}$ under the diagonal map.  We note $S_n \cong A / 7^n A$ and so is finite.  By Proposition \ref{strong-approx-7}, we have that $\bigcup_n S_n$ is dense in $\widehat{A}$.

  The fact that some $S_n$ is $\delta$-dense in $\widehat{A}$ follows from this and a routine compactness argument.  Since the action of $\Theta$ permutes the finite set $S_n$, we may choose $m$ such that $\theta_5^m$ and $\theta_{13}^m$ fix $S_n$ pointwise, as required.
\end{proof}

\section{Proof of the rationality lemma}
\label{sec-rationality}

\subsection{The growth estimate, i.e.~Lemma \ref{furstenberg-density}}

Most of the hard work in running the proof of the \tttt{} Theorem is in proving the analogue of Lemma \ref{furstenberg-density}, which was used to show Corollary \ref{furstenberg-zero-limit} which in turn was crucial for the inductive step of the main argument.

This proof is still highly technical.  However, as we regard it as the core of the argument, we will place it here rather than in an appendix.

In fact we split into two results, the former corresponding to the assumption ``$0$ is a limit point of $Y$'', the latter to ``$Y$ has a rational limit point''.

\begin{lemma}
  \label{zero-limit-lemma}
  Suppose $Y \subseteq \widehat{A}$ is a closed, $\Theta$-invariant set with $0$ as a non-Archimedean limit point.  Then $Y = \widehat{A}$.

  The same holds replacing $\Theta$ by $\Theta^{(m)}$ for any $m > 0$.
\end{lemma}

\begin{lemma}
  \label{torsion-limit-lemma}
  Let $Y \subseteq \widehat{A}$ be a closed, $\Theta$-invariant set, and $y = \jmath_\C(w) + t \in \widehat{A}$, where $w \in \C$ and $t$ is a torsion point.  Suppose $y$ is a non-Archimedean limit point of $Y$.  Then the image $Y(1)$ of $Y$ under the evaluation map, is all of $\T$.

  Again, the same holds replacing $\Theta$ by $\Theta^{(m)}$.
\end{lemma}

\begin{remark}
  As mentioned above, we know of no reason why the weaker conclusion of Lemma \ref{torsion-limit-lemma} could not in principle be strengthened to $Y = \widehat{A}$, and this is the sole cause of the correspondingly weaker conclusion to Lemma \ref{infinitary-rationality-lemma}.
  
  A similarly weak conclusion to Lemma \ref{zero-limit-lemma} would \emph{not} suffice for the argument.
\end{remark}

We need one more standard fact, the proof of which is again in Appending \ref{appA}.
\begin{proposition}
  \label{topological-generation}
  Suppose $p > 2$ and $u \in \Z_p^\times$ is not a root of unity.  Then $J_0 = \bar{\{u^r :\, r \ge 0\}}$ is a finite index subgroup of $\Z_p^\times$.
\end{proposition}

Before embarking on the proof of Lemma \ref{zero-limit-lemma}, we introduce an abuse of notation: for the sake of visual clarity, where it is unambiguous to do so, we suppress the explicit use of $\imath_{\C}$, $\imath_{\Q_5}$ and $\imath_{\Q_{13}}$, implicitly treating $\Q(i)$ as a subset of $\C, \Q_5, \Q_{13}$ respectively.

\begin{proof}[Proof of Lemma \ref{zero-limit-lemma}]
  If $0 \in Y'$, we can choose a sequence $x_n$ in $Y$ such that $x_n \rightarrow 0$ is a non-Archimedean limit.  We may identify $x_n$ with points $(z_n, a_n, b_n)$ of the fundamental domain $[-1/2, -1/2)^2 \times \Z_5 \times \Z_{13}$ of $\widehat{A}$.  Then, $z_n, a_n, b_n$ converge individually to $0$.

  Since the limit is non-Archimedean, we know that $a_n$ and $b_n$ cannot both be eventually zero.  Suppose $a_n$ is not eventually zero -- the other case is symmetric.  Passing to an appropriate subsequence we may assume $a_n \ne 0$ for all $n$.

  Let $J_0$ denote the closure of $\left\{{\theta_{13}}^{m r} :\, r \in \Z_{\ge 0}\right\} \subseteq \Z_5^\times$.  Since $\theta_{13}$ is not a root of unity in $\Z_5^\times$, we have that $J_0$ is a (closed) finite index subgroup of $\Z_5^\times$, by Proposition \ref{topological-generation}. Let $J$ denote the subgroup of $\Q_5^\times$ spanned by $J_0$ and ${\theta_5}^m$; so $J$ has finite index in $\Q_5^\times$ (this can be seen directly).  Finally, pick $H$ to be some coset of $J$ such that $a_n \in H$ infinitely often; passing to an appropriate subsequence, assume $a_n \in H$ for all $n$.  Set $H_0 = H \cap \Z_5^\times$ (a coset of $J_0$).

  By Proposition \ref{strong-strong-approx}, it suffices to show $\jmath(0, H, 0) \subseteq Y$, as $\jmath(0, H, 0)$ is dense in $\widehat{A}$.  So, fix any $\alpha \in H$; we will show that $\jmath_{\Q_5}(\alpha) \in Y$.

  Fix a $\delta > 0$.  Our aim is to find an $n$ and integers $r, s \ge 0$ such that $d_{\widehat{A}}({\theta_{13}}^{m r} {\theta_5}^{m s} (x_n), \jmath(0, \alpha, 0)) \le \delta$. Since ${\theta_{13}}^{m r} {\theta_5}^{m s} (x_n)$ lies in $Y$ (by $\Theta^{(m)}$-invariance) and since $Y$ is closed and $\delta$ is arbitrary, this again suffices.

  The distance $d_{\widehat{A}}({\theta_{13}}^{m r} {\theta_5}^{m s} (x_n), \jmath(0, \alpha, 0))$ is (at most) the sum of the three terms $|{\theta_5}^{m s} {\theta_{13}}^{m r} z_n|_\C$, $|{\theta_5}^{m s} {\theta_{13}}^{m r} a_n - \alpha|_5$  and $|{\theta_5}^{m s} {\theta_{13}}^{m r} b_n|_{13}$.  The hard work is approximating $\alpha$ by ${\theta_5}^{m s} {\theta_{13}}^{m r} a_n$ to make the $\Q_5$ contribution small; but we need to guarantee that the $\C$ and $\Q_{13}$ terms stay small as we do so.  Of course, by choosing $n$ large we can make $z_n$ and $b_n$ as small as we like; but some consideration is needed of the fact that $r$ and $s$ may themselves depend on $n$.
  
  First we control the contribution from the $\C$ term.  This is straightforward: for $n$ sufficiently large, we may insist that $|z_n|_\C \le \delta / 3$, and we observe that $|\theta z_n|_\C = |z_n|_\C$ for any $\theta \in \Theta$.
  
  Now we consider the $\Q_{13}$ term.  Since this becomes large as $r$ increases, we will need some bound on $r$ that is independent of $n$.  Specifically, we choose $N$ such that the set $\{{\theta_{13}}^{m r} :\, 0 \le r \le N \}$ is $|\alpha|_5^{-1}\, \delta / 3$-dense in $J_0$  (again we have invoked compactness, this time of $J_0$, to get the appropriate uniformity in this statement) and commit to only considering $r$ in the range $0 \le r \le N$.
  Now -- again taking $n$ sufficiently large -- we can assume that $|b_n|_{13} \le 13^{-N m}\, \delta / 3$; equivalently, that $|{\theta_{13}}^{m r} b_n|_{13} \le \delta / 3$ for all $0 \le r \le N$.

  Finally, we consider the $\Q_5$ term, which is the heart of the matter.  Again taking $n$ large enough, we may assume $|a_n|_5 \le |\alpha|_5$.  At this stage we fix some particular $n$, large enough in all the ways described above.  It remains only to fix $r$ and $s$.

  For ${\theta_5}^{m s} {\theta_{13}}^{m r} a_n$ to approximate $\alpha$, it should have the same valuation, which is controlled by $s$. So, we are forced to take $s = \left(v_5(a_n) - v_5(\alpha)\right) / m \ge 0$ (recalling $a_n$ is non-zero).  As $a_n$ and $\alpha$ lie in the same coset of $J$ by assumption, we have that $s$ is an integer and $u := \alpha / ({\theta_5}^{m s} a_n) \in J_0$.

  Now we must choose $r$ so that ${\theta_{13}}^{m r}$ approximates $u$.  Indeed, by our choice of $N$ there is an $r$ in the range $0 \le r \le N$ such that $|{\theta_{13}}^{m r} - u|_5 \le |\alpha|_5^{-1}\, \delta / 3$, whence
  \[
    |{\theta_{13}}^{m r} {\theta_5}^{m s} a_n - \alpha|_5 = |{\theta_{13}}^{m r} - u|_5\, |\alpha|_5 \le \delta / 3 \ .
  \]

  Putting everything together, we finally get:
  \begin{align*}
    d_{\widehat{A}}\left({\theta_{13}}^{m r} {\theta_5}^{m s}(x_n),\, \jmath(0, \alpha, 0) \right) & = |z_n|_\C + |{\theta_{13}}^{m r} {\theta_5}^{m s} a_n - \alpha|_5 + |{\theta_{13}}^{m r} {\theta_5}^{m s} b_n|_{13} \\
  &\le \delta / 3 + \delta / 3 + \delta / 3 = \delta
  \end{align*}
  as required.
\end{proof}

We deduce a simple corollary.

\begin{cor}
  \label{zero-limit-cor}
  Suppose $Y$ is closed, $\Theta^{(m)}$-invariant and has a point $x$ as a  non-Archimedean limit point.  Let $V = \bar{\Theta^{(m)} (x)}$ be the orbit closure of $x$.  Then $Y - V = \widehat{A}$.
\end{cor}
\begin{proof}
  Note $Y - V$ is closed and $\Theta^{(m)}$-invariant.  If $x_n \rightarrow x$ is a non-Archimedean limit in $Y$ then $x_n - x \rightarrow 0$ is a non-Archimedean limit in $V - Y$, so we apply Lemma \ref{zero-limit-lemma} to $V - Y$.
\end{proof}

Clearly as $V$ becomes larger this conclusion becomes weaker.  In the case of Lemma \ref{torsion-limit-lemma}, $V$ is fairly small by virtue of Proposition \ref{complex-by-torsion-lemma} \ref{enum:cbf-iff-cpt}, and the conclusion of the corollary is strong enough.  The following lemma shows why.

\begin{lemma}
  \label{big-circles}
  Suppose $y = \jmath_\C(w) + t \in \widehat{A}$ where $w \in \C$ and $t$ is torsion.  Pick any $m > 0$, and let $V$ be the orbit closure of $y$ under $\Theta^{(m)}$.  If $|w|_\C \ge 1/2$ then $V(1) = \T$.
\end{lemma}
\begin{proof}
  By Corollary \ref{torsion-nearly-implies-periodic} we can choose $\theta_0 \in \Theta$ such that $\theta_0(t)$ is periodic.  If $x$ is periodic then so is $\theta(x)$ for any $\theta \in \Theta$, and so we can adjust $\theta_0$ to lie in $\Theta^{(m)}$.  Choose $m' > 0$ such that $m | m'$ and $\Theta^{(m')}$ fixes $\theta_0(t)$.

  We have $\Theta^{(m)} (y) \supseteq \Theta^{(m')} (\theta_0(y)) = \theta_0(t) + \Theta^{(m')} (\jmath_\C(\theta_0\, w))$.  Since $\Theta^{(m')}$ contains all the powers of some irrational rotation (such as ${\theta_5}^{m'}$) and $\theta_0 \, w$ is just some complex number, $\{\theta\, \theta_0\, w \,:\, \theta \in \Theta^{(m')}\}$ is dense on the circle $\{ z \in \C \,:\, |z| = |w| \}$.  So, $V$ contains $S = \theta_0(t) + \jmath_\C\left(\{ z \in \C \,:\, |z| = |w| \}\right)$.

  So, $V(1)$ contains $S(1) = t(\theta_0) + \{ \Re(z) \,:\, z \in \C,\,  |z| = |w| \}$.  The latter is all of $\T$ provided $|w| \ge 1/2$, as required.
\end{proof}

\begin{proof}[Proof of Lemma \ref{torsion-limit-lemma}]
  Let $V$ denote the orbit closure of $y$; clearly all points of $V$ are of the form $\jmath_\C(w') + t'$ where $|w'| = |w|$ and $t'$ is torsion.  By Corollary \ref{zero-limit-cor} we can find $z \in Y$ and $v = \jmath_\C(w') + t' \in V$ such that $z - v = \jmath_\C(100 + |w|)$.  So $z = \jmath_\C(w' + 100 + |w|) + t'$ has the form required by Lemma \ref{big-circles}, and clearly $|w' + 100 + |w|| \ge 100 \ge 1/2$ as $|w'| = |w|$.  Since $Y$ contains the orbit closure $\bar{\Theta^{(m)}(z)}$ , the conclusion of Lemma \ref{big-circles} applied to $z$ gives the result.
\end{proof}

\subsection{The main argument}

Finally, we finish the proof of the infinitary rationality lemma (Lemma \ref{infinitary-rationality-lemma}).  This part of the argument follows the proof of Theorem \ref{tttt} almost line for line, which is possible only by repeatedly invoking results from Section \ref{sec-solenoid}.

\begin{proof}[Proof of Lemma \ref{infinitary-rationality-lemma}]
  Let $Y \subseteq \widehat{A}$ be a closed invariant subset, and assume $Y$ is \emph{not} a finite union of complex balls with centers at torsion points.  Indeed, if it were, then either Lemma \ref{big-circles} gives $Y(1) = \T$ or $Y$ is contained in a set of the form $U = \bigcup_{i=1}^k \bar{B}_{1/2}^\C(u_i)$ as required.
  
  By Proposition \ref{complex-by-torsion-lemma} \ref{enum:cbf-implies-cbt} and Proposition \ref{infinite-limit-lemma}, $Y'$ is non-empty, and it is closed and invariant by Proposition \ref{limit-is-closed}.  Also, we may assume $Y'$ does not contain any points of the form $x = \jmath_\C(w) + t$ with $w \in \C$ and $t$ torsion, as this would imply $Y(1) = \T$ by Lemma \ref{torsion-limit-lemma}.

  Fix an arbitrary $\delta > 0$ and choose $m$ and $S \subseteq \widehat{A}$ according to Proposition \ref{many-periodic-points}; that is, $S$ consists of fixed points of $\Theta^{(m)}$ and is $\delta$-dense in $\widehat{A}$.  Assume for convenience that $0 \in S$, and write $S = \{a_0, a_1, \dots, a_{\ell - 1} \}$ where $a_0 = 0$.  For $0 \le k < \ell$, define
  $X_k := (Y' - a_0) \cap \dots \cap (Y' - a_k)$.
  Note that $X_0 = Y'$, $X_k$ are nested, and furthermore
  $
    X_{k+1} = X_k \cap \left( Y' - a_{k+1} \right)
  $
  which implies (inductively) that $X_k$ is closed and $\Theta^{(m)}$-invariant for all $k$ (recalling that $\Theta^{(m)}$ fixes the $a_i$).
  
  We claim inductively that $X_k$ is non-empty; since $X_k \subseteq Y'$ and $Y'$ contains no points $x = \jmath_\C(w) + t$ where $w \in \C$ and $t$ is torsion, this immediately implies $X_k$ is not contained in a finite union of complex balls (Proposition \ref{complex-by-torsion-lemma}\ref{enum:cbf-iff-cpt}).  The base case $k = 0$ is by assumption; suppose this holds for $X_k$.  By Proposition \ref{infinite-limit-lemma}, $X_k - X_k$ has $0$ as a non-Archimedean limit point, and hence so does $Y' - X_k$.  By  Lemma \ref{zero-limit-lemma}, $Y' - X_k = \widehat{A}$ and in particular we can find $x \in X_k$, $y \in Y'$ such that $y - x = a_{k+1}$.  I.e.~$x \in \left(Y' - a_{k+1}\right)$ and as $x \in X_k$ we deduce $x \in X_{k+1}$, completing the induction.

  Hence for any $x \in X_{\ell-1}$, we have $\{x + a_0,\, x + a_1,\, \dots,\, x + a_{\ell - 1}\} \subseteq Y' \subseteq Y$.  So $Y$ contains a translate of $S$ and hence is $\delta$-dense in $\widehat{A}$.  But $\delta > 0$ was arbitrary and $Y$ is closed; hence $Y = \widehat{A}$ and \emph{a fortiori} $Y(1) = \T$.
\end{proof}

\appendix

\section{Proofs of auxiliary results}
\label{appA}

Here we provide proofs of some of the less standard and more technical results of Section \ref{sec-solenoid}.

\subsection{Torsion points and periodic points}
\begin{proof}[Proof of Proposition \ref{periodic-points}]
  Consider \ref{enum:torsion-equiv}.  Unwrapping the definitions, we have that $n\, x = 0$ for $n \in \Z$ if and only if there is some $r \in A$ such that $n\, z = \imath_\C(r)$, $n\, a = \imath_{\Q_5}(r)$, $n\, b = \imath_{\Q_{13}}(r)$.  Setting $q = r / n$ gives (\ref{torsion-cond} $\Rightarrow$ \ref{torsion-rationality-cond}) and choosing $n \in \Z$ such that $n\, q \in A$ gives (\ref{torsion-rationality-cond} $\Rightarrow$ \ref{torsion-cond}).
  
  For (\ref{orbit-finite-cond} $\Rightarrow$ \ref{torsion-cond}), observe that by pigeonhole there exist $\phi \ne \psi \in \Theta$ such that $\phi(x) = \psi(x)$, whence $(\phi - \psi) (x) = 0$ and so, multiplying by $a \in \Z[i]$ such that $a (\phi - \psi) \in \Z$, we deduce $x$ is torsion.
  
  Finally, for (\ref{torsion-cond}, \ref{torsion-rationality-cond} $\Rightarrow$ \ref{orbit-finite-cond}) we note that $n\, x = 0 \Rightarrow n \theta (x) = 0$ for $\theta \in \Theta$, so it suffices to check that, for given $n$, the set
  $
    \{ x \in \widehat{A} \,:\, n\, x = 0 \}
  $
  is finite.  This is actually a group, and by (\ref{torsion-rationality-cond}) is isomorphic under $\imath^\Delta$ to the quotient
   $ 
    \{ q \in \Q(i) \,:\, n\, q \in A \} / A
   $,
   which in turn is isomorphic (under $q \mapsto n q$) to $A / n A$.  That this last group is finite is a standard fact.
  \\[\baselineskip]
  We now turn to \ref{enum:periodic-equiv}.  For (\ref{fixed-point-cond} $\Rightarrow$ \ref{periodic-rationality-cond}), we write $\phi = (\theta_{5}\, \theta_{13})^m$ and suppose $\phi(x) = x$.  That means there exists $r \in A$ such that $\imath_\C(\phi - 1) z = \imath_\C(r)$, $\imath_{\Q_5}(\phi - 1) a = \imath_{\Q_5}(r)$, $\imath_{\Q_{13}}(\phi - 1) b = \imath_{\Q_{13}}(r)$, so taking $q = r / (\phi - 1)$ it suffices to check that $|q|_{\PP_5},\,  |q|_{\PP_{13}} \le 1$.  Note $|r|_{\PP_5},\, |r|_{\PP_{13}} \le 1$ (since $r \in A$) and $|\phi|_{\PP_5},\, |\phi|_{\PP_{13}} < 1$ since $m > 0$, so $|\phi - 1|_{\PP_5},\, |\phi - 1|_{\PP_{13}} = 1$.
  
  For (\ref{periodic-rationality-cond} $\Rightarrow$ \ref{fixed-point-cond}), we observe as above that the orbit $\Theta(x)$ is contained in
  \[
    \{ x \in \widehat{A} \,:\, n\, x = 0 \} \cong A / n\, A
  \]
  for some $n \in \Z$, which crucially can now be taken to be coprime to $65$.  It follows that the corresponding multiplicative action of $\Theta$ on the ring $A / n\, A$ is now invertible, i.e.~the image of $\Theta$ lies in $(A / n\, A)^\times$. Hence we can choose $m$ to be $|(A / n\, A)^\times|$.
\end{proof}

\subsection{Unions of complex balls}
\begin{proof}[Proof of Proposition \ref{complex-by-torsion-lemma}]
  We prove \ref{enum:cbf-implies-cbt}.  Let $x_1, \dots, x_k \in Y$ be such that $\bigcup_i \bar{B}_1^\C(x_i)$ covers $Y$.   It suffices to show each $x_i$ has the form $x_i = y_i + \jmath_\C(z_i)$ where $y_i$ is torsion and $z_i \in \C$.

  By pigeonhole, there exist distinct $\theta_1, \theta_2 \in \Theta^{(m)}$ such that $\theta_1(x_i)$ and $\theta_2(x_i)$ lie in the same ball $\bar{B}_1^\C(x_j)$.  Hence $\theta_1(x_i) - \theta_2(x_i) \in \jmath_{\C}(\C)$.  Choosing $r \in \Z[i]$ such that $r (\theta_1 - \theta_2) = n \in \Z$ we obtain $n\, x_i \in \jmath_{\C}(\C)$ as required.
  \\[\baselineskip]
  Now consider \ref{enum:cbf-iff-cpt}. $\Rightarrow$ follows trivially from \ref{enum:cbf-implies-cbt}: $\bar{\Theta^{(m)}(x)}$ is contained in a finite union of complex balls, is closed and invariant, and hence is contained in a finite union of complex balls with torsion centers, which \emph{a fortiori} implies $x$ itself has the specified form.

  For $\Leftarrow$, we invoke Proposition \ref{periodic-points} \ref{enum:torsion-equiv}(\ref{torsion-cond} $\Rightarrow$ \ref{orbit-finite-cond}) and note the trivial fact that the orbit closure of a point $\jmath_\C(w)$, $w \in \C$ is contained in (in fact, is equal to) $\jmath_\C\left(\{ z \in \C \,:\, |z| = |w| \}\right)$, which is certainly contained in some $\bar{B}_R^\C(0)$.
\end{proof}

\needspace{3\baselineskip}
\subsection{Non-Archimedean limits}
\begin{proof}[Proof of Proposition \ref{infinite-limit-lemma}]
  By the hypothesis on $Y$, we can pick a sequence $x_i \in Y$ such that $x_i - x_j \notin \bar{B}_{1}^\C(0)$ for all $i \ne j$.  Since $Y$ is compact metric, passing to some subsequence we have $x_i \rightarrow x$ for some $x \in Y$.
  
  It would suffice to show that $x_i - x \notin \bar{B}_{1/2}^\C(0)$ for all $i$, as then $x_i \rightarrow x$ and $x_i - x \rightarrow 0$ would be non-Archimedean limits in $Y$ and $Y - Y$ respectively.

  In fact this holds for all $i$ with at most one exception: if $x_i - x \in \bar{B}_{1/2}^\C(0)$ and $x_j - x \in \bar{B}_{1/2}^\C(0)$ for $i \ne j$ then $x_i - x_j \in \bar{B}_1^\C(0)$, contradicting the choice of the $x_i$.  Deleting the exceptional index (if there is one) gives the result.
\end{proof}

\subsection{Density of $\C$, $\Q_5$ and $\Q_{13}$ in $\widehat{A}$}

%
%
%
%

We need a standard fact for the proof of Proposition \ref{strong-strong-approx}.
\begin{proposition}
  \label{finite-index-mult-subgroup-lemma}
  Any finite index multiplicative subgroup of $\Z_p^\times$ contains a subgroup $(1 + p^n \Z_p, \times)$ for some $n$.
\end{proposition}
\begin{proof}
  There is an isomorphism $(1 + p \Z_p, \times) \longleftrightarrow (\Z_p, +)$ given by the $p$-adic $\log$ and $\exp$ maps (for $p > 2$); see e.g. \cite{serre}[Chapter II, Section 3, Proposition 8].

  So if $J \le \Z_p^\times$ has finite index then $J \cap (1 + p \Z_p)$ has finite index in $1 + p \Z_p$, which corresponds to a finite index subgroup of $(\Z_p, +)$ under the isomorphism; those are all $(p^n \Z_p, +)$ for some $n$, which correspond to $(1 + p^{n+1} \Z_p, \times)$ under the isomorphism the other way.
\end{proof}

\begin{cor}
  \label{cosets-are-open}
  Let $J$ be a finite index subgroup of $\Q_p^\times$.  There exists some $\eta > 0$ depending only on $J$, such that if $x, y \in \Q_p^\times$ and $|x - y|_p / |x|_p \le \eta$ then $x$ and $y$ are in the same coset of $J$.
\end{cor}
\begin{proof}
  We require $y / x \in J$.  By Proposition \ref{finite-index-mult-subgroup-lemma} applied to the finite index subgroup $J_0 = J \cap \Z_p^\times \le \Z_p^\times$, there is some $\eta > 0$ such that if $a \in \Z_p^\times$ and $|1 - a|_p \le \eta$ then $a \in J_0$.  So -- provided $\eta \le 1$ -- we have $|1 - a|_p \le \eta \Rightarrow a \in J_0$ for all $a \in \Q_p^\times$.  But $|1 - y / x|_p = |x - y|_p / |x|_p$ and the result follows.
\end{proof}

We isolate one further result from the proof of Proposition \ref{strong-strong-approx}.
\begin{lemma}
  \label{small-coset-element-finder}
  Let $J,\, H$ be as in the statement of Proposition \ref{strong-strong-approx}.  For any $\mu, \nu > 0$ we can find an $r \in A$ such that $|r|_\C,\, |r|_{\bar{\PP_{13}}} \le \mu$, $|r|_{\bar{\PP_5}} \ge \nu$ and $\imath_{\Q_5}(r) \in H$.  The same holds with $5$ and $13$ exchanged.
\end{lemma}
\begin{proof}
  Let $x_1, \dots, x_k \in A$ be coset representatives for $J$; i.e.~$\bigcup_i x_i\, J = \Q_5^\times$.  (This is possible by Corollary \ref{cosets-are-open} and the fact that $A$ is dense in $\Q_p^\times$.) Let $C = \max \left\{ |x_i|_\C,\, |x_i|_{\bar{\PP_{13}}},\, 1 / |x_i|_{\bar{\PP_5}} \right\}$.  Pick any $y \in A$ such that $|y|_\C,\, |y|_{\bar{\PP_{13}}} < 1$ and $|y|_{\bar{\PP_5}} > 1$ (say, $y = \bar{\PP_{13}} / \bar{\PP_5}^{10}$).  Then choose $n$ large enough that $|y^n|_\C,\, |y^n|_{\bar{\PP_{13}}} \le \mu / C$ and $|y^n|_{\bar{\PP_5}} \ge C\, \nu$.  Now take $i$ such that $x_i\, y^n \in H$, and observe $r = x_i\, y^n$ has the desired properties.
\end{proof}

\begin{proof}[Proof of Proposition \ref{strong-strong-approx}]
  Let $x = \jmath(z, a, b) \in \widehat{A}$, with $(z, a, b)$ in the fundamental domain (say).  We take $\delta > 0$ arbitrary, and choose $q \in A$ such that $|\imath_\C(q) - z|_\C,\, |\imath_{\Q_{13}}(q) - b|_{13} \le \delta / 4$ (by Proposition \ref{strong-approx}).

  If $a - \imath_{\Q_5}(q) \in H$ we are happy, as then
  \[
    d_{\widehat{A}}((z, a, b),\, (0,\, a - \imath_{\Q_5}(q),\, 0)) = d_{\widehat{A}}((z, a, b),\, (\imath_\C(q),\, a,\, \imath_{\Q_{13}}(q)) \le \delta / 2 \ .
  \]
  If not, the above results allow us to perturb $q$ to some $q - r$ so that this holds (i.e.~$a - \imath_{\Q_5}(q - r) \in H$), without affecting the other properties of $q$ too much.
  
  Specifically, take $\eta > 0$ as in Corollary \ref{cosets-are-open}, and $r$ as in Lemma \ref{small-coset-element-finder} with parameters $\mu = \delta / 4$ and $\nu = 10 |a - \imath_{\Q_5}(q)|_5 / \eta$.  We claim $a - \imath_{\Q_5}(q - r) \in H$.  Indeed, as $\imath_{\Q_5}(r) \in H$ by construction and $|a - \imath_{\Q_5}(q)|_5 / |\imath_{\Q_5}(r)|_5 \le \eta / 10$, this follows by Corollary \ref{cosets-are-open}.  Finally,
  \begin{align*}
    d_{\widehat{A}}((z, a, b),\, (0, a - \imath_{\Q_5}(q - r), 0)) &= d_{\widehat{A}}((z, a, b),\, (\imath_\C(q - r),\, a,\, \imath_{\Q_{13}}(q - r)) \\
                                                                   & \le |z - \imath_\C(q)|_\C + |r|_\C + |\imath_{\Q_{13}}(q) - b|_{\bar{\PP_{13}}} + |r|_{\bar{\PP_{13}}} \le \delta
  \end{align*}
  which, as $\delta$ was arbitrary, completes the proof.
\end{proof}



\subsection{Proof of Proposition \ref{topological-generation}}

\begin{proof}[Proof of Proposition \ref{topological-generation}]
  Using \cite{serre}[{Chapter II, Section 3, Proposition 7}], we write $u = \xi \, v$ where $\xi$ is a $(p-1)^{\mathrm{st}}$ root of unity and $v \in 1 + p \Z_p$.  By assumption $v \ne 1$.  It suffices to show $\bar{\{v^{(p-1)r} :\, r \ge 0\}}$ is a finite index subgroup of $(1 + p \Z_p, \times)$.  Using the $p$-adic $\log$/$\exp$ isomorphism (\cite{serre}[Chapter II, Section 3, Proposition 8]) this reduces to showing that for $x \in \Z_p$ non-zero, the closure $\bar{\{(p-1) r\, x \,:\, r \ge 0\}}$ is a finite index subgroup of $(\Z_p, +)$; and in fact it is precisely $p^{v_p(x)} \Z_p$.
\end{proof}

\bibliography{pyjama}{}
\bibliographystyle{halpha}

\end{document}